\documentclass{amsart} 

\usepackage[utf8]{inputenc} 
\usepackage{mathtools}
\usepackage{amssymb}
\usepackage{amsmath}
\usepackage{amsthm}
\usepackage{listings}
\usepackage{physics}
\usepackage{calrsfs}
\usepackage{enumerate}
\usepackage{tikz-cd}
\usepackage{mathrsfs}
\usepackage{bm}
\usepackage{esvect}
\usepackage{enumitem}

 \usepackage[parfill]{parskip}
 
\newtheorem{theorem}{Theorem}[section]
\newtheorem{lemma}[theorem]{Lemma}
\newtheorem{proposition}[theorem]{Proposition}
\newtheorem{corollary}[theorem]{Corollary}

\newtheorem{conjecture}[theorem]{Conjecture}
\newtheorem{definition}[theorem]{Definition}
\theoremstyle{remark}
\newtheorem{remark}[theorem]{Remark}

\numberwithin{equation}{section}

\DeclareMathOperator{\Var}{\mathrm{Var}}

\renewcommand{\r}{\mathbb{R}}
\renewcommand{\c}{\mathbb{C}}
\newcommand{\e}{\mathbb{E}}
\newcommand{\z}{\mathbb{Z}}
\newcommand{\h}{\mathbb{H}}
\newcommand{\q}{\mathbb{Q}}

\newcommand{\lr}[1]{\left (   {#1} \right )}

\newcommand{\psl}{\mathrm{PSL}}
\newcommand{\vol}{\mathrm{vol}(\Gamma \backslash \h^3)}

\newcommand{\GA}{\Gamma_{\a}' \backslash \Gamma}
\newcommand{\GH}{\Gamma \backslash \h^3}

\newcommand{\Ltwo}{L^2(\GH, \chi_{\epsilon})}
\newcommand{\Le}{\widetilde{L}(\epsilon)}
\newcommand{\inprod}[2]{\left \langle  {#1} , {#2} \right \rangle}
\newcommand{\Hcusp}{H^1_{\text{cusp}}}
\newcommand{\sa}{\sigma_{\mathfrak{a}}}
\newcommand{\sigmab}{\sigma_{\mathfrak{b}}}
\renewcommand{\a}{\mathfrak{a}}
\renewcommand{\b}{\mathfrak{b}}
\newcommand{\rr}{\langle r \rangle}
\newcommand{\T}{T_{\mathfrak{a} \mathfrak{b}}}
\renewcommand{\L}{L_{\mathfrak{a} \mathfrak{b}}}
\newcommand{\OK}{\mathcal{O}_K}
\newcommand{\Gn}{\Gamma_0(\mathfrak{n})}

\usepackage{geometry} 
\title[Distribution of modular symbols in $\h^3$]{Distribution of modular symbols in $\h^3$} 
\author[Petru Constantinescu]{Petru Constantinescu}
\address{Department of Mathematics\\University College London\\
	25 Gordon Street, London, UK, WC1H  0AY}
\email{petru.constantinescu.17@ucl.ac.uk}

\begin{document}

\begin{abstract}
   We introduce a new technique for the study of the distribution of modular symbols, which we apply to congruence subgroups of Bianchi groups. We prove that if $K$ is a quadratic imaginary number field of class number one and $\mathcal{O}_K$ its ring of integers, then for certain congruence subgroups of $\mathrm{PSL}_2(\mathcal{O}_K)$, the periods of a cusp form of weight two obey asymptotically a normal distribution. These results are specialisations from the more general setting of quotient surfaces of cofinite Kleinian groups where our methods apply. We avoid the method of moments. Our new insight is to use the behaviour of the smallest eigenvalue of the Laplacian for spaces twisted by modular symbols. Our approach also recovers the first and the second moment of the distribution. 
\end{abstract}

\maketitle

\section{Introduction}

Mazur and Rubin \cite{MR} proposed the study of arithmetic statistics of modular symbols in order to gain information about the non-vanishing of the central value $L(E, \chi, 1)$, where $E / \q$ is an elliptic curve and $\chi$ a primitive character.  By the conjectures of Birch--Swinnerton-Dyer, this is related to studying when there is excess rank
\begin{equation*}
    \text{rank} E(L) > \text{rank}E(\q) \ ,
\end{equation*}
where $L / \q$ is an abelian extension. Motivated by this, the study of the distribution of modular symbols became a very active area; see the work of Petridis--Risager \cite{gafa}, \cite{petridis_geodesics}, \cite{petridis}, Diamantis--Hoffstein--K\i ral--Lee \cite{diamantis}, Lee--Sun \cite{LS}, Bettin--Drappeau \cite{BD} and Nordentoft \cite{asbjorn}. In this work, we investigate the distribution of modular symbols associated to an imaginary quadratic field. 

Let $K$ be a quadratic number field of class number one, $\OK$ its ring of integers and $\mathfrak{n}$ a non-zero ideal of $\OK$. In a series of papers \cite{cremona_thesis}, \cite{cremona2}, \cite{cremona}, Cremona uses modular symbols to study the arithmetic correspondence between isogeny classes of elliptic curves defined over $K$ of conductor $\mathfrak{n}$ and Hecke cusp forms of weight 2 for the congruence subgroup $\Gamma_0(\mathfrak{n})$. More precisely, the Hasse--Weil $L$-function $L(E,s)$ of an elliptic curve $E$ and the $L$-function $L(F,s)$ attached to a cusp form $F$ are conjectured to be the same as part of the {\lq Langlands philosophy\rq}. Modular symbols are given by central values $L(F,\psi,1)$, where $\psi$ is an additive twist, and they can be used to compute numerically the central value $L(F,1)$, which agrees with the value $L(E,1)$ predicted by the Birch--Swinnerton-Dyer conjecture. We prove that when $\mathfrak{n}$ is a square-free ideal of $\OK$ and $F$ a newform of weight 2 and level $\mathfrak{n}$, modular symbols coming from $F$  obey asymptotically the standard normal distribution when ordered and normalised appropriately. 

We develop a new method to obtain distribution results for modular symbols. While still making use of the spectral theory of Eisenstein series as in the work of Petridis--Risager, we apply the perturbation theory on character varieties to obtain significantly easier proofs. Also, instead of using the method of moments for proving convergence in distribution, we make use of the moment generating function and the Berry--Esseen inequality to obtain the limiting distribution with almost optimal error terms. Furthermore, our approach can naturally recover the first and second moments of the distribution and has the advantage that it can be naturally extended to modular symbols in $\h^3$.

To describe our results, we briefly review modular symbols for holomorphic cusp forms on $\Gamma_0(q)$. Let $f$ be a weight 2 holomorphic cusp form for $\Gamma_0(q)$ and $\alpha= \Re(f(z) dz)$ the associated real-valued, cuspidal one-form. Since the cusps are parametrised by $\q$, we write
\begin{equation*}
    \rr=\int_{i \infty}^r \alpha.
\end{equation*}
The path can be taken as the vertical line connecting $r \in \q$ to $\infty$. \\

We begin by stating the conjectures of Mazur and Rubin. We define the usual mean and variance for fixed level $c$
\begin{equation*}
    \mbox{E}(f,c)=\frac{1}{\phi(c)} \sum_{\substack{a \text{ mod } c \\ (a,c)=1}} \langle a/c \rangle , \quad  \Var (f,c)=\frac{1}{\phi(c)} \sum_{\substack{a \text{ mod } c \\ (a,c)=1}} (\langle a/c \rangle - \mbox{E}(f,c))^2 \ .
\end{equation*}

\begin{conjecture}[Mazur--Rubin]
\label{conj 2}
Fix $f \in S_2(\Gamma_0(q))$, where $q$ is a positive integer. Then there exists a constant $C_f$ and, for each divisor $d$ of $q$, constants $D_{f,d}$, such that
\begin{equation*}
    \lim_{\substack{c \to \infty \\ (c,q)=d}} (\mathrm{Var}(f,c)-C_f \log c) = D_{f,d} \ .
\end{equation*}
\end{conjecture}
The constant $C_f$ is called the {\it variance slope} and the constant $D_{f,d}$ the {\it variance shift}.

Moreover, they conjectured that modular symbols obey a normal distribution:

\begin{conjecture}[Mazur--Rubin]
\label{conj 3}
The limiting distribution of the data
\begin{equation*}
    \frac{\langle a / c \rangle}{(C_f \log c + D_{f,d})^{1/2}}  , \quad \text{with } (c,q)=d, \ a \in (\z/ c \z)^*
\end{equation*}
is the standard normal distribution.
\end{conjecture}

We now describe the set-up for the general case of cofinite groups $\Gamma$ of $\mbox{PSL}_2(\r)$, as in the work of Petridis--Risager. Let $\a$ and $\b$ be two cusps (not necessarily equivalent) with scaling matrices $\sa$ and $\sigmab$. We define general modular symbols as
\begin{equation*}
    \rr_{\a \b} = \int_{\b}^{\sa r} \alpha  , 
\end{equation*}
where $\alpha$ is a harmonic 1-form and
\begin{equation*}
    r \in \T(X)=\left \{ \frac{a}{c} \text{ mod }1 \ , \ \begin{pmatrix} a & b \\ c & d\end{pmatrix} \in \Gamma_{\infty} \backslash \sa^{-1} \Gamma \sigmab / \Gamma_{\infty} \ , \ 0<c<X\right \}  .
\end{equation*}

Petridis--Risager obtain the following average results of Conjectures \ref{conj 2} and \ref{conj 3}. 
\begin{theorem}[Petridis--Risager \cite{petridis}]
\label{pr}
There exist explicit constants $C_f$, $D_{f, \a \b}$ such that
\begin{itemize}
\item[(a)] {\it (Normal distribution)} The values of 
\begin{equation*}
    \T(X) \to \r, \quad \frac{a}{c} \mapsto  \frac{\langle a/c \rangle }{\sqrt{C_f \log c}}
\end{equation*}
have asymptotically a standard normal distribution as $X \to \infty$.
\item[(b)] {\it (Second moment)} As $X \to \infty$,
\begin{equation*}
     \frac{\sum_{r \in \T (X)}\rr_{\a \b}^2}{\# \T (X)} = C_f \log X + D_{f, \a \b} + o(1).
\end{equation*}

\end{itemize}    
\end{theorem}

Petridis--Risager worked with the spectral theory of automorphic forms. In particular, they make use of Eisenstein series twisted by modular symbols, introduced by Goldfeld \cite{goldfeld1} \cite{goldfeld2}. They obtain results for $\Gamma$ a general cofinite Fuchsian group and $f$ a cusp form of weight 2. Using dynamical properties of the Gauss map, Lee--Sun \cite{LS} obtain normal distribution for the case where $f \in S_2(\Gamma_0(N))$, while Bettin--Drappeau \cite{BD} get results for general weight $k$, but with $\Gamma=\Gamma_0(1)$. Using the spectral theoretical methods introduced by Petridis--Risager, Nordentoft \cite{asbjorn} obtains normal distribution for central values of additively twisted $L$-functions associated to cusp forms of general weight $k$ and level $N$. Our method uses only the twisted Eisenstein series and perturbation theory on character varieties. This we apply to $\h^3$, but can also be worked out for $\h^2$ and give an easier proof of Theorem \ref{pr} with explicit and good error terms.

Here is a statement for our results. There is a natural action of $ \psl_2(\c)$ on $\h^3$ via isometries. Let $\Gamma \leq \psl_2(\c)$ be a cofinite discrete subgroup. For each cusp $\a$, we denote by $\Gamma_{\a}'$ the set of parabolic elements in $\Gamma$ that fix $\a$. Then there exists a lattice $\Lambda_{\a} \leq \c$ such that 
\begin{equation*}
    \sa^{-1} \Gamma_{\a}' \sa =  \left \{ \begin{pmatrix} 1 & \lambda \\ 0 & 1 \end{pmatrix} \ : \  \lambda \in \Lambda_{\a} \right \} \ .
\end{equation*}

We note that we require this extra notation since, unlike the two dimensional case, we only know that $\Gamma_{\a}'$ is a subgroup of finite index of the stabilizer subgroup $\Gamma_{\a}$ and that for two cusps $\a$ and $\b$, the period lattices $\Lambda_{\a}$ and $\Lambda_{\b}$ may be different.

Now, for $\a, \mathfrak{b}$ two cusps for $\Gamma$ (not necessarily distinct), we define 
\begin{equation*}
    R_{\a \b}(X)=\left \{ r=\frac{a}{c} \text{ mod } \Lambda_{\a} \ , \ \begin{pmatrix} a & b \\ c & d\end{pmatrix} \in \sa^{-1} \Gamma_{\a}' \sa \backslash \sa^{-1} \Gamma \sigmab / \sigmab^{-1} \Gamma_{\b}' \sigmab \ , \ 0<|c|<X\right \} \ .
\end{equation*}

We prove the following theorem.

\begin{theorem}
    \label{main theorem}
    Let $\alpha$ be a real-valued, $\Gamma$-invariant, cuspidal one-form. 
    
    \begin{itemize}
    \item[(a)] {\it (Normal distribution)} For every $a,b \in [-\infty, \infty]$ with $a \leq b$, and any $\epsilon>0$, for $X$ large enough, 
    \begin{equation*}
        \frac{\# \left \{ \ r \in R_{\a \b}(X) \ , \ \frac{\rr_{\a \b}}{\sqrt{C_{\alpha} \log X}} \in [a,b] \right \}}{\# R_{\a \b}(X)} = \frac{1}{\sqrt{2 \pi}} \int_a^b \exp \lr{-\frac{t^2}{2}} dt  + O_{\epsilon} \left ( (\log X)^{-1/2+\epsilon}  \right ) ,
    \end{equation*}
    where
    \begin{equation}
    \label{C_alpha}
        C_{\alpha} = \frac{4 \| \alpha \|_2^2}{\vol} \ .
    \end{equation}
    \item[(b)] {\it (First moment)} There exists a constant $\delta>0$ such that
    \begin{equation*}
        \frac{\sum_{r \in R_{\a \b}(X)} \rr_{\a \b}}{\# R_{\a \b}(X)} =  \int_{\b}^{\a} \alpha + O \lr{X^{- \delta}} \ .
    \end{equation*}
    \item[(c)] {\it (Second moment)} There exists an explicit constant $D_{\alpha, \a \b}$, called the variance shift, and a constant $\delta>0$ such that
    \begin{equation*}
        \frac{ \sum_{r \in R_{\a \b}(X)} \rr_{\a \b}^2}{\# R_{\a \b}(X)} = C_{\alpha} \log X + D_{\alpha, \a \b} + O\lr{X^{- \delta}}, \quad \hbox{as } X \to \infty \ .
    \end{equation*}
    \end{itemize}
\end{theorem}

\begin{remark}
    The error term in Theorem \ref{main theorem}(a) is expected to be optimal up to $\epsilon$, see \cite{BD}. It seems to be difficult to obtain a good error term using  the method of moments approach.
\end{remark}

\begin{remark}
Theorem \ref{main theorem}(b) is a generalisation of \cite[Cor. 7.3]{petridis} with $x=1$, where Petridis--Risager obtain stronger results about first moment with additional restrictions on the set $R_{\a \b}(X)$.
\end{remark}

\begin{remark}
    We do not obtain an explicit value for $D_{\alpha, \a \b}$, but we can write it in terms of the coefficients of a certain Taylor expansion, see \eqref{Dalpha} for more details. For the case of $\h^2$, the variance shift was explicitly calculated in \cite{petridis}.
\end{remark}

We obtain the following corollary for imaginary quadratic number fields. Let $K$ be a quadratic imaginary field of class number one and $\mathfrak{n}$ a square-free ideal. Let $F\in S_2(\Gamma(\mathfrak{n}))$ be a cuspidal newform of weight 2  and level $\mathfrak{n}$, which is a vector-valued function $F: \h^3 \to \c^3$. For $r \in K$, we define the modular symbol
\begin{equation*}
    \rr= \int_{i \infty}^r F \cdot\beta \in \r \ ,
\end{equation*}
where $\beta$ is a specific fixed basis for the invariant 1-forms. We rigorously introduce these objects in Section 7.

\begin{corollary}
    \label{quadratic}
    Let K be a quadratic number field of class number one. Let $\mathfrak{n} \vartriangleleft \mathcal{O}_K$ a square-free ideal with generator $\langle n \rangle = \mathfrak{n}$ and $F \in S_2(\Gamma_0(\mathfrak{n}))$. For $\mathfrak{d} \vert \mathfrak{n}$, set $$Q_{\mathfrak{d}}(X)=\{a/c \ | \  a \in (\mathcal{O}_K / \langle c \rangle )^{\times} \ , \  \langle c,n \rangle =\mathfrak{d} \ , \  0< |c|<X \}.$$
     \begin{itemize}
     \item[(a)] There exists a constant $C_F$ such that the data
    \begin{align*}
            K\cap Q_{\mathfrak{d}}(X) & \to \r, \quad
            \frac{a}{c} \mapsto \frac{\langle a/c \rangle }{\sqrt{C_F \log X}}
        \end{align*}
        has asymptotically a standard normal distribution. 
        
    \item[(b)] There exists a constant $D_{F, \mathfrak{d}}$ such that
        \begin{equation*}
            \frac{1}{ |Q_{\mathfrak{d}}(X)|  } \sum_{a/c \in Q_{\mathfrak{d}} (X)} \left \langle \frac{a}{c}\right \rangle^2 = C_F \log X + D_{F, \mathfrak{d}}+ o(1) \ .
        \end{equation*}
       \end{itemize}
\end{corollary}

\begin{remark}
We provide explicit value for $C_F$ in terms of the Petersson norm of $F$ and our base quadratic imaginary field $K$, see  \eqref{C_F}.
\end{remark}

The structure of the paper is as follows. In Section 2 we introduce the basic properties of the space $\GH$. We highlight the elementary properties of modular symbols associated to cuspidal one-forms. 

In Section 3 we study the Eisenstein series and Poincar\'{e} series twisted by modular symbols. We introduce the generating series $L_{\a \b}(s, \epsilon)$ and obtain some of their essential analytic properties. We also provide upper bounds for modular symbols. 

In Section 4 we study the perturbation theory of the space $\Ltwo$, where $\chi_{\epsilon}$ is a unitary character given by modular symbols. We obtain Taylor expansions for the smallest eigenvalue of the Laplacian $\lambda_0(\epsilon)$ and for $s_0(\epsilon)$, the first pole of $L_{\a \b}(s, \epsilon)$. Moreover, we study the behaviour of the residue of $L_{\a \b}(s, \epsilon)$ at $s_0(\epsilon)$.

In Section 5 we relate the moment generating function for the distribution of modular symbols to our generating series $L_{\a \b}(s, \epsilon)$. We recover the first two moments of the distribution. In addition, we show that $R_{\a \b}$ is equidistributed in the period lattice $\Lambda_{\a}$.

In Section 6 we prove that modular symbols are normally distributed. We use the Berry--Esseen inequality and the perturbation theory results developed earlier. 

In Section 7 we obtain results for congruence subgroups of $\textrm{PSL}_2(\mathcal{O}_K)$, where $K$ is a quadratic imaginary number field of class number one. We relate modular symbols to special values of $L$-functions coming from newforms of weight 2 and level $\mathfrak{n}$, where $\mathfrak{n}$ is a square-free ideal of $\OK$. We develop some properties of these $L$-functions.

\section{The geometry of the quotient space $\GH$}

\subsection{Notation}

We refer to \cite[Chapters 1-2]{egm} for a valuable exposition of the geometry of the hyperbolic 3-space and of the groups acting on it. We define the three-dimensional hyperbolic space $\h^3$ as
\begin{align*}
    \label{definitionofH}
    \h^3 := \c \times (0, \infty)=\{ (z,y) \ | \ z \in \c, y>0\}= \{ (x_1,x_2,y) \ | \ x_1,x_2 \in \r, \  y >0\} \ .
\end{align*}

We denote the points in $\h^3$ by
\begin{equation*}
    P=(z,y)=z+yj, \quad \text{where }  z=x_1+ix_2, \ j=(0,0,1) \ .
\end{equation*}

We equip $\h^3$ with the hyperbolic metric coming from the line element:
\begin{align}
    \label{linemetric}
    ds^2=\frac{dx_1^2+dx_2^2+dy^2}{y^2} \ .
\end{align}

The volume element is given by
\begin{equation*}
    \label{volumemetric}
    dv=\frac{dx_1 dx_2 dy}{y^3} .
\end{equation*}

The hyperbolic Laplace--Beltrami operator is given by
\begin{equation}
    \label{laplaceoperator}
    \Delta = y^2 \lr{\pdv[2]{}{x_1} + \pdv[2]{}{x_2} + \pdv[2]{y}} - y \pdv{y} \ .
\end{equation}

The group $\psl_2(\c)$ acts on $\h^3$ via isometries. The action of $\gamma=\begin{pmatrix} a & b \\ c & d  \end{pmatrix} \in \psl_2(\c)$ is given by
\begin{equation}
    \label{matrixaction}
    (z,y) \mapsto \lr{ \frac{(az+b) \overline{(cz+d)} + a \overline{c} y^2}{|cz+d|^2 + |c|^2 y^2} \ , \ \frac{y}{|cz+d|^2+|c|^2 y^2} } \ .
\end{equation}

Let $\Gamma \leq \hbox{PSL}_2(\c)$ be any cofinite Kleinian group with cusps. The theory of such objects is thoroughly developed in \cite[Chapter 2]{egm}. Let $\a \in \mathbb{P}^1 (\c)$ be a cusp for $\Gamma$ with scaling matrix $\sa \in \hbox{PSL}_2(\c)$ such that $\sa \infty = \a$. We let $\Gamma_{\a} = \{ \gamma \in \Gamma \ : \ \gamma \a = \a \}$ be the stabilizer of $\a$ in $\Gamma$. We define 
\begin{equation*}
    \Gamma_{\a}'= \Gamma_{\a} \cap \sa  \left \{  \begin{pmatrix} 1 & b \\ 0 & 1\end{pmatrix} \ : \ b \in \c \right \} \sa^{-1} \ .
\end{equation*}
We note that $\Gamma_{\a}'$ consists of the parabolic elements in $\Gamma_{\a}$ together with $I$.

There exists a lattice $\Lambda_{\a} \leq \c$ such that 
\begin{equation*}
    \sa^{-1} \Gamma_{\a}' \sa =  \left \{ \begin{pmatrix} 1 & \lambda \\ 0 & 1 \end{pmatrix} \ : \  \lambda \in \Lambda_{\a} \right \} \ .
\end{equation*}

We let $\mathcal{P}_{\a}$ be a period parallelogram for $\Lambda_{\a}$ with Euclidean area $| \mathcal{P}_{\a}|$. 

We define $\Lambda_{\a}^{\circ}$ the dual lattice of $\Lambda_{\a}$:
\begin{equation}
    \label{dual lattice}
    \Lambda_{\a}^{\circ}= \{ \mu \in \c \ : \ \inprod{\mu}{\lambda} \in \z \text{ for all } \lambda \in \Lambda_{\a} \} \ ,
\end{equation}
where $\inprod{\cdot}{\cdot}$ is the usual scalar product on $\r^2=\c$.

Since $\Gamma$ is a Kleinian group, there exists a constant $c_{\a \b}>0$ defined by
\begin{equation}
    \label{cab}
    c_{\a \b} := \min \left \{|c| \ : \ \begin{pmatrix} a & b \\ c& d \end{pmatrix} \in \sa^{-1} \Gamma \sigmab, \ c \neq 0 \right \} \ .
\end{equation}
Say $\a_1, \cdots, \a_h \in \mathbb{P}^1 (\c)$ are representatives for the $\Gamma$-classes of cusps. For $Y>0$, we define the cuspidal sectors
\begin{equation*}
    \mathcal{F}_{\a_i}(Y)=  \sigma_{\a_i} \{ z + y j \ : \ z \in \mathcal{P}_{\a_i}  \ , \ y \geq Y\} \ .
\end{equation*}
Then for $Y_0$ large enough, there exists a fundamental domain $\mathcal{F}$ which we can write as a disjoint union
\begin{equation}
    \label{Y_0}
    \mathcal{F}=\mathcal{F}_0 \cup \mathcal{F}_{\a_1}(Y_0) \cup \cdots \cup \mathcal{F}_{\a_h}(Y_0) \ ,
\end{equation}
where $\mathcal{F}_0$ is a compact set.

We denote by $\T$ a system of representatives $\begin{pmatrix} * & * \\ c & * \end{pmatrix}$ of the double cosets in
\begin{equation*}
    \sa^{-1} \Gamma_{\a}' \sa \backslash \sa^{-1} \Gamma \sigmab / \sigmab^{-1} \Gamma_{\b}' \sigmab
\end{equation*}
with $c \neq 0$ and 
\begin{equation*}
    \T(X)= \left \{ \begin{pmatrix} * & * \\ c & * \end{pmatrix} \in \T \ : \ 0 < |c| \leq X \right \} \ .
\end{equation*}

Also, we define 
\begin{equation*}
    R_{\a \b} : = \left \{ \frac{a}{c} \mbox{ mod } \mathcal{P}_{\a} \ : \  \begin{pmatrix} a & b \\ c & b \end{pmatrix} \in \T\right \} \ .
\end{equation*}

\begin{lemma}
\label{T to R}
The map
\begin{align*}
    \T &\to R_{\a \b} \\
    \gamma & \mapsto \gamma \infty \mod \mathcal{P}_{\a}
\end{align*}
is $[\Gamma_{\b}: \Gamma_{\b}']$-to-one.
\end{lemma}

\begin{proof}
We follow the lines of \cite[Proposition 2.2]{petridis} or \cite[p. 50]{iwaniec}, where it is shown that the map is one-to-one in the two-dimensional case. Let $\gamma, \gamma' \in \T$ with
\begin{equation*}
    \gamma=\begin{pmatrix} a & b \\ c & d \end{pmatrix} \quad \textrm{and} \quad \gamma'=\begin{pmatrix} a' & b' \\ c' & d' \end{pmatrix}
\end{equation*}
and $r=\gamma \infty$, $r'=\gamma' \infty$. We may assume $r, r' \in \mathcal{P}_{\a}$. Then the matrix $\gamma''=\gamma'^{-1} \gamma \in \sigmab^{-1} \Gamma \sigmab$ has lower left entry $c''=-ac'+a'c$.

If $c''\neq 0$, then
\begin{equation*}
    |-r+r'|=\left | \frac{c''}{c c'}\right | >0 \ .
\end{equation*}
Therefore $r \neq r'$, hence $r \not \equiv r' \textrm{ mod } \mathcal{P}_{\a}$.

If $c''=0$, then $r=r'$ and $\gamma'' \in (\sigmab^{-1} \Gamma \sigmab)_{\infty}=\sigmab^{-1} \Gamma_{\b} \sigmab $. Since we assume $\gamma, \gamma' \in \T$, there are $[\sigmab^{-1} \Gamma_{\b} \sigmab: \sigmab^{-1} \Gamma_{\b}' \sigmab]=[\Gamma_{\b}: \Gamma_{\b}']$ possible choices for $\gamma''$.
\end{proof}

\subsection{Construction of modular symbols}

We denote by $\h^*:= \h^3 \cup \c \cup \{ \infty \}$ the extended upper half-space and consider the compactified quotient space $X_{\Gamma}= \Gamma \backslash \h^*$. If $A,B \in \h^*$ are $\Gamma$-equivalent, i.e. there exists some $\gamma \in \Gamma$ such that $B=\gamma(A)$, then the family of smooth paths from $A$ to $B$ in $\h^*$ determines a unique homology class in $H_1(X_{\Gamma}, \z)$. In fact, the class depends only on $\gamma$ and we have the surjective map 
\begin{align*}
    \Phi : \Gamma  &\to H_1 (X_{\Gamma}, \z), \quad
    \gamma \mapsto \{\infty, \gamma \infty \} 
\end{align*}
which induces the canonical isomorphism 
\begin{equation*}
     H_1 (X_{\Gamma}, \z) \cong \Gamma / [\Gamma,\Gamma] \ . 
\end{equation*}

We consider the de Rham cohomology group $H^1(X_{\Gamma}, \c)$ and inside of it we have  $H^1_c(X_{\Gamma}, \c)$ consisting of cohomology classes represented by forms of compact support. Every member of $H^1_c(X_{\Gamma}, \c)$ has a harmonic representative. We provide a sketch argument showing that $ H_1(X_{\Gamma}, \c) $ and $H^1_c(X_{\Gamma}, \c)$ are dual to each other.

Note that in general $X_{\Gamma}$ may not be a manifold, since $\Gamma$ may contain elements of finite order ($X_{\Gamma}$ is called an orbifold). However, it is a result of Selberg \cite[p. 482]{selberg} that if $\Gamma < \mbox{GL}_n(\c)$ is a finitely generated subgroup, then $\Gamma$  has a torsion free subgroup $\Gamma'$ of finite index. Then $X_{\Gamma'}$ is a manifold and the finite quotient group $\bar{\Gamma}:= \Gamma / \Gamma'$ acts on it. We have the exact Poincar\'{e} pairing between homology and cohomology for $X_{\Gamma'}$
\begin{align*}
    H_1(X_{\Gamma'}, \c) \times H^1_c(X_{\Gamma'}, \c)  \to \c, \quad 
    (C, \alpha) & \mapsto \int_C \alpha \ .
\end{align*}
In this duality, if we restrict to forms invariant under $\bar{\Gamma}$, we recover $H^1_c(X_{\Gamma}, \c)$ and can show that there is also an exact duality between $ H_1(X_{\Gamma}, \c) $ and $H^1_c(X_{\Gamma}, \c)$. For more details, see \cite[p. 43]{cremona_thesis}.

\begin{definition}
\label{cuspidal 1-form}
A harmonic 1-form $\alpha=f_1 dx_1 +f_2 dx_2 + f_3 dy$
on $\GH$ is a cuspidal 1-form if
\begin{enumerate}
    \item $\alpha$ is rapidly decreasing at all cusps;
    \item for each cusp $\a$ and $y \geq 0$, $$\int_{\mathcal{P}_{\a}}f_{\a,i} dx_1 dx_2 = 0 \ , \quad i=1,2,3 \ ,$$
where $\sa^{*}\alpha=f_{\a,1} dx_1 + f_{\a, 2}d x_2 + f_{\a,3} dy$.
\end{enumerate}
\end{definition}

As in \cite{sarnak_cusp_forms}, we denote the space of cuspidal 1-forms by $\Hcusp(X_{\Gamma}, \c)$. We note that any cuspidal form is cohomologous to a form of compact support, i.e. if $\alpha$ is a cusp form, there exists $\widetilde{\alpha} \in H^1_c(X_{\Gamma}, \c)$ such that 
\begin{equation*}
    \int_{\Phi(\gamma)} \alpha =  \int_{\Phi(\gamma)} \widetilde{\alpha} , \quad \text{for all } \gamma \in \Gamma \ ,
\end{equation*}
and we have the isomorphism
\begin{equation*}
    \Hcusp(X_{\Gamma}, \c) \simeq H^1_c(X_{\Gamma}, \c) \ .
\end{equation*}

A detailed construction of the above isomorphism can be found in \cite[Proposition 2.1]{gafa}. With this in mind, for $\gamma \in \Gamma$ and $\alpha \in \Hcusp(X_{\Gamma}, \c)$, we define the modular symbol $\inprod{\gamma}{\alpha}$ as 
\begin{equation}
    \label{modular symbol def}
    \inprod{\gamma}{\alpha} := \int_{\Phi(\gamma)}{\alpha} =\int_{P_0}^{\gamma P_0} \alpha 
\end{equation}
for any $P_0 \in \h^*$. From this definition, we can easily see that, for $\gamma_1, \gamma_2 \in \Gamma$,
\begin{equation*}
    \inprod{\gamma_1 \gamma_2}{\alpha}= \int_{P}^{\gamma_1 \gamma_2 P} \alpha = \int_{P}^{\gamma_2 P} \alpha + \int_{\gamma_2 P}^{\gamma_1 \gamma_2 P} \alpha = \inprod{\gamma_1}{\alpha} + \inprod{\gamma_2}{\alpha} \ . 
\end{equation*}

We note that if $\alpha$ is a cuspidal form, then for any parabolic $\gamma \in \Gamma$,
\begin{equation*}
    \inprod{\gamma}{\alpha}=\int_{P_0}^{\gamma P_0} \alpha = 0 \ .
\end{equation*}
In particular, $\inprod{\gamma}{\alpha}=0$, for all $\gamma \in \Gamma_{\a}'$, for all cusps $\a$.

We remark that our definition for the modular symbol $\inprod{\gamma}{\alpha}$ agrees with the previous definition $\rr_{\a \b}$. Indeed, if $\gamma \in \sa^{-1 }\Gamma \sigmab$  with $r=\gamma \infty$, then
\begin{equation}
    \label{relation}
    \rr_{\a \b}=\int_{\b}^{\sa \gamma \infty} \alpha = \int_{\b}^{\sa \gamma \sigmab^{-1} \b} \alpha = \inprod{\sa \gamma \sigmab^{-1}}{\alpha} \ . 
\end{equation}

If $\alpha \in \Hcusp(X_{\Gamma}, \c)$ is real-valued, we have a family of unitary characters $\chi_{\epsilon}: \Gamma \to S^1$ defined by 
\begin{equation}
\label{character}
    \chi_{\epsilon}(\gamma): = \exp \lr{2 \pi i \epsilon \inprod{\gamma}{\alpha}} \ .
\end{equation}

If $\alpha, \beta \in \Hcusp(X_{\Gamma}, \c)$ with $\alpha=f_1 dx_1 + f_2 dx_2 + f_3 dy$ and $\beta= g_1 dx_1 + g_2 dx_2 + g_3 dy$, we define the pointwise inner-product
\begin{equation}
    \label{pointwise ip}
    [\alpha,\beta] : = y^2 (f_1 \overline{g_1} + f_2 \overline{g_2} + f_3 \overline{g_3}) \ .
\end{equation}
Since $\alpha$ and $\beta$ are $\Gamma$-invariant 1-forms, one can see that $[\alpha,\beta]$ is a $\Gamma$-invariant function from $\h^3$ to $\c$. In particular, since $\alpha$ is rapid decreasing in the cusps, we conclude that $[\alpha, \alpha]$ is bounded on $\h^3$, which in turn implies that
\begin{equation}
    \label{f_i bound}
    |f_i(P)| \ll \frac{1}{y}, \quad \text {for all }P \in \h^3, i=1,2,3.
\end{equation}

Now, for $\alpha, \beta \in \Hcusp(X_{\Gamma}, \c)$, we define the Petersson inner product
\begin{equation}
\label{in_prod_1_forms}
    \inprod{\alpha}{\beta} :=\int_{\GH} [\alpha,\beta] dv \ ,
\end{equation}
and the $L^2$-norm
\begin{equation}
    \| \alpha \|_2^2 := \inprod{\alpha}{\alpha} \ .
\end{equation}

\section{Generating series for modular symbols}

In this section we define a generating series for modular symbols $L_{\a \b}(s, \epsilon)$. This we relate to the twisted Eisenstein series and Poincar\'{e} series by characters and derive some of their essential analytic properties.

\subsection{Twisted Eisenstein series by modular symbols}
We define the twisted Eisenstein series
\begin{equation}
    \label{definition of Eisensten series}
    E_{\a}(P, s, \epsilon) = \sum_{\gamma \in \GA} \overline{\chi_{\epsilon}(\gamma)} y( \sa^{-1} \gamma P)^s \ ,
\end{equation}
where $\chi_{\epsilon}$ is defined as in \eqref{character}.

The theory of such series is developed in \cite[Chapter 3]{egm}. They are absolutely convergent for $\Re(s)>2$. In the area of absolute convergence they satisfy
\begin{align*}
     E_{\a}(\gamma P, s, \epsilon) &= \chi_{\epsilon}(\gamma)  E_{\a}(P, s, \epsilon) \ , \\
     - \Delta E_{\a} (P, s, \epsilon) & =  s(2-s) E_{\a}(P, s, \epsilon) \ .
\end{align*}

We note that the function $P \mapsto E_{\a}(\sigma_{\b}P, s, \epsilon)$ is invariant under the action of the lattice $\Lambda_{\b}$ corresponding to $\sigma_{\b}^{-1} \Gamma_{\b}' \sigmab=(\sigmab^{-1} \Gamma \sigmab)_{\infty}'$. We would like to write a Fourier expansion with respect to the dual lattice $\Lambda_{\b}^{\circ}$. With this in mind, for $\mu_1 \in \Lambda_{\a}^{\circ}$, $\mu_2 \in \Lambda_{\b}^{\circ}$, we define the twisted generating series by
\begin{equation}
    \label{lseries}
    \L(s, \mu_1, \mu_2, \epsilon) :=   \sum_{\gamma \in \T} \frac{\overline{\chi_{\epsilon}(\sa \gamma \sigmab^{-1})} e \lr{\inprod{\mu_1} {\frac{a}{c}}+ \inprod{\mu_2}{\frac{d}{c}}}}{|c|^{2s}} \ ,
\end{equation}
where the sum is over $\gamma = \begin{pmatrix} a & b \\ c & d \end{pmatrix} \in \T$. If $\mu_1=\mu_2=0$, we just denote $L_{\a \b}(s,0,0, \epsilon)= :L_{\a \b}(s,\epsilon).$

We quote \cite[Theorem 3.4.1]{egm} to obtain Fourier expansion of $E_{\a}(\sigmab P,s, \epsilon)$:
\begin{align}
    E_{\a}(\sigmab P, s ,\epsilon)=\delta_{\a \b}[\Gamma_{\a}: \Gamma_{\a}'] y^s + \phi_{\a \b}(s, \epsilon) y^{2-s} +  \sum_{0 \neq \mu \in \Lambda_{\b}^{\circ}}
     |\mu|^{s-1} \phi_{\a \b}(s, \mu, \epsilon) \ y \ K_{s-1}(2 \pi |\mu| y) \ e(\inprod{\mu}{z})  \ , 
      \label{fourier expansion}
\end{align}
where
\begin{equation}
    \label{phi definition}
    \phi_{\a \b}(s, \epsilon): = \frac{\pi}{|\mathcal{P}_{\b}| (s-1)} \L(s, \epsilon) , \quad 
    \phi_{\a \b}(s, \mu, \epsilon) := \frac{2 \pi^s}{|\mathcal{P}_{\b}| \Gamma(s)} \L(s, 0, \mu, \epsilon) 
\end{equation}
and $K$ denotes the $K$-Bessel function.

We now quote the theory from \cite[chapter 6.1]{egm}. We have to modify it slightly since we consider twisted Eisenstein series, so we follow the steps in Selberg's G\"{o}ttingen lecture notes \cite[p. 638-654]{selberg}. If $\a_1, \cdots, \a_h \in \mathbb{P}^1 (\c)$ are the inequivalent cusps for $\GH$, we define
\begin{equation*}
    E_i(P,s, \epsilon) : = \frac{1}{[\Gamma_{\a_i}:\Gamma_{\a_i}' ]} E_{\a_i}(P,s, \epsilon) \quad \text{and} \quad \phi_{ij}( s, \epsilon) =\frac{1}{[\Gamma_{\a_i}:\Gamma_{\a_i}'] } \phi_{\a_i \a_j}( s, \epsilon) \ .
\end{equation*}

We let
\begin{equation*}
    \mathcal{E}(P, s , \epsilon): = \begin{pmatrix} E_1(P, s, \epsilon) \\ \vdots \\ E_h(P, s, \epsilon)\end{pmatrix} \quad \text{and} \quad \Phi(s, \epsilon) : = (\phi_{ij}(s, \epsilon)) \ .
\end{equation*}
We call $\Phi$ the scattering matrix. Then both $\mathcal{E}(P, s, \epsilon)$ and $\Phi(s, \epsilon)$ have meromorphic continuation to all of $\c$. The following functional equation is satisfied:
\begin{equation*}
    \mathcal{E}(P, 2-s, \epsilon)=\Phi(2-s, \epsilon) \mathcal{E}(P, s, \epsilon) \ .
\end{equation*}

 Also, poles of $\mathcal{E}(P,s, \epsilon)$ occur only where $\Phi(s,\epsilon)$ has poles and vice versa. In the region $\Re s > 1$, there are only finitely many simple poles, and they are on the interval $1 <s \leq 2$ of the real line. If $1 < \sigma \leq 2$ is a pole of $E_{\a}(P,s,\epsilon)$, we define 
 \begin{equation}
     u_{\a, \sigma}(P, \epsilon)= \Res_{s= \sigma} E_{\a} (P, s, \epsilon) \ .
 \end{equation}
 
 We denote by $\Ltwo$ the space of all square-integrable functions $f$ over $\GH$ that satisfy $f(\gamma P)= \chi_{\epsilon}(\gamma) f(P)$, for all $\gamma \in \Gamma$. Then $u_{\a, \sigma}(P, \epsilon ) \in \Ltwo$ and moreover
 \begin{equation}
     (\Delta + \sigma(2-\sigma))u_{\a, \sigma}(\cdot, \epsilon) = 0 \ .
 \end{equation}
 
 We study the spectral theory of $\Ltwo$ in Section \ref{Laplacian}. The spectrum of $-\Delta$ on $\Ltwo$ contains a finite number of discrete eigenvalues in $[0,1)$, call them $0 \leq \lambda_0(\epsilon) \leq \lambda_1(\epsilon) \leq \cdots \leq \lambda_k(\epsilon)<1$. Then $E_{\a}(P,s, \epsilon)$ is meromorphic for $\Re(s)>1$ and has possible poles at $s_j(\epsilon)$ corresponding to $\lambda_j(\epsilon)$, so that $s_j(\epsilon)(2-s_j(\epsilon)) = \lambda_j(\epsilon)$.

  \subsection{Twisted Poincar\'{e} series by modular symbols}
We now introduce the twisted Poincar\'{e} series, extending the definition of Sarnak in \cite{sarnak_3_manifolds}. We will use them to obtain an integral representation for the series $L_{\a \b}(s, 0, \mu, \epsilon)$ and to find the residue of $L_{\a \b}(s, 0, \mu, 0)$ at $s=2$.

For $\mu \in  \Lambda_{\a}^{\circ}$, we define
\begin{equation}
    \label{poincare}
    E_{\a, \mu} (P, s, \epsilon): =\sum_{\gamma \in \GA}   \overline{\chi_{\epsilon}(\gamma)} y(\sa^{-1} \gamma P)^s e^{- 2 \pi |\mu| y(\sa^{-1} \gamma P)} e(\langle z(\sa^{-1} \gamma P), \mu \rangle) \ .
\end{equation}
We observe that for $\Re(s)>2$, the series converges absolutely, since it is certainly dominated by the Eisenstein series. Also, since the function $y(\sa^{-1} P)^s e^{- 2 \pi |\mu| y(\sa^{-1}  P)} e(\langle z(\sa^{-1}  P), \mu \rangle)$ is $\Gamma_{\a}'$-invariant, it follows that  $E_{\a, \mu} (P, s, \epsilon)$ satisfies
\begin{equation*}
     E_{\a, \mu}(\gamma P, s, \epsilon) = \chi_{\epsilon}(\gamma)  E_{\a, \mu}(P, s, \epsilon)
\end{equation*}
and that $E_{\a, \mu} (\sigmab P, s, \epsilon)$ is $\Lambda_{\b}$-invariant. Additionally, it is easy to check that for $\Re(s)>2$ and $\mu \neq 0$,
\begin{equation}
\label{PoincareL2}
E_{\a, \mu}(P, s, \epsilon) \in L^2(\Gamma \backslash \h, \chi_\epsilon) .
\end{equation}
An easy computation shows that
\begin{equation}
    (\Delta+s(2-s))E_{\a, \mu}(P,s, \epsilon)= 2 \pi |\mu| (1-2s) E_{\a, \mu}(P, s+1, \epsilon) ,
\end{equation}
which can be rewritten as
\begin{equation}
    \label{PoincareResolvent}
    E_{\a, \mu}(P,s, \epsilon)= 2 \pi |\mu| (1-2s)R(s(2-s), \epsilon)(E_{\a, \mu}(P, s+1, \epsilon)),
\end{equation}
where $R(\lambda, \epsilon)$ is the resolvent of $\Delta$ on $\Ltwo$ at $\lambda$. We have that $R(s(2-s), \epsilon)$ is meromorphic for $\Re(s)>1$ and has possible poles at $s_j(\epsilon)$. Hence, from \eqref{PoincareL2} and \eqref{PoincareResolvent}, it follows that $E_{\a, \mu}(P,s, \epsilon)$ may be analytically continued to $\Re(s)>1$, with possible poles at $s_j(\epsilon)$.

Next, we want to use the Poincar\'{e} series to obtain an integral representation for the generating series $L_{\a \b}(s, 0, \mu,  \epsilon)$.

\begin{lemma}
\label{integral rep}
    Let $\mu \in \Lambda_{\b}^{\circ} \setminus \{ 0 \}$ and $\Re(s), \Re(w)>2$. Then we have the integral representation
    \begin{align*}
        L_{\a \b}(s, 0, \mu, \epsilon) = \frac{|\mathcal{P}| (4 \pi |\mu|)^{w-1} }{2 \pi^{s+1/2} |\mu|^{s-1}} \frac{\Gamma(s)\Gamma(w-1/2)}{\Gamma(w+s-2)\Gamma(w-s)} \int_{\Gamma \backslash \h^3} E_{\a}(P, s , \epsilon) \overline{E_{\b, \mu} (P, \overline{w}, \epsilon)} dv \ .
    \end{align*}
\end{lemma}

\begin{proof}
We use a standard unfolding technique together with \eqref{fourier expansion} and \eqref{poincare} to obtain
\begin{align*}
      \int_{\Gamma \backslash \h^3} E_{\a}(P, s , \epsilon) \overline{E_{\b, \mu} (P, \overline{w}, \epsilon)} d v
     =& \int_{0}^{\infty} \int_{\mathcal{P_{\b}}} E_{\a}(\sigmab P, s , \epsilon) y^{w} e^{- 2 \pi |\mu| y} e(- \langle z, \mu \rangle ) \frac{dx_1 dx_2 dy}{y^3} \\
     =& \int_0^{\infty} y^{w} e^{-2 \pi |\mu| y}  |\mathcal{P}_{\b}| |\mu|^{s-1} \phi_{\a \b}(s, \mu,  \epsilon) \ y \ K_{s-1}(2 \pi |\mu| y) \frac{dy}{y^3} \\
     =& L_{\a \b}(s, 0, \mu,  \epsilon) \frac{2 \pi^s}{\Gamma(s)} |\mu|^{s-1} \frac{\sqrt{\pi}}{(4 \pi |\mu|)^{w-1}} \frac{\Gamma(w+s-2) \Gamma(w-s)}{ \Gamma(w-1/2)} ,
\end{align*}
where we have used \cite[p. 205]{iwaniec} for the integral of the Bessel function. 
\end{proof}

\begin{remark}
    Similar to the above calculation, it follows that, for $\mu \in \Lambda_{\b}^{\circ} \setminus \{ 0 \}$, 
    \begin{equation*}
        \int_{\GH} E_{\b, \mu}(P, s, 0)  dv =0 \ . \\
    \end{equation*}
\end{remark}

Next we want to use Lemma \ref{integral rep} to find the analytic properties of $L_{\a \b}(s, 0, \mu ,  0)$ at $s=2$. 

\begin{lemma}
    \label{residues}
    For $\mu \in \Lambda_{\b}^{\circ}$, the series $L_{\a \b} (s, 0, \mu, \epsilon)$ admits meromorphic continuation to $s \in \c$. At $s=2$, $L_{\a \b}(s,0)$ has a pole with residue
    \begin{equation*}
    \Res_{s=2} L_{\a \b}(s, 0) =  \frac{| \mathcal{P}_{\a}|| \mathcal{P}_{\b}| [\Gamma_{\a}: \Gamma_{\a}'] }{ \pi \vol }
\end{equation*}
while for $\mu \neq 0$, $L_{\a \b }(s, 0, \mu,  0)$ is holomorphic at $s=2$.
\end{lemma}

\begin{proof}
Since the Eisenstein series $E_{\a}(P, s, \epsilon)$ admits meromorphic continuation to $s \in \c$, its Fourier coefficients admit meromorphic continuation as well. Hence from \eqref{fourier expansion} and \eqref{phi definition}, we obtain meromorphic continuation for $L_{\a \b}(s, 0, \mu , \epsilon)$.

We know that $E_{\a}(P, s, 0)$ has a simple pole at $s=2$ and it follows from the Maa\ss--Selberg relations in $\h^3$ \cite[p. 110]{egm} that
\begin{equation}
    \label{eisenstein residue}
    \Res_{s=2} E_{\a}( \sigmab P,s,0)= \frac{|\mathcal{P}_{\a}| [\Gamma_{\a}: \Gamma_{\a}']}{\vol} \ .
\end{equation}
The conclusion follows from relating $L_{\a \b}(s,0)$ to the $0$-th Fourier coefficient of $E_{\a}(P, s, 0)$, as it can be seen from \eqref{fourier expansion} and \eqref{phi definition}.

Now, when $\mu \neq 0$, then we know that $L_{\a \b}(s, 0, \mu , 0)$ has at most one simple pole at $s=2$. Using the integral representation from Lemma \ref{integral rep}, this residue would have $\langle 1, E_{\b, \mu}(P, \overline{w}, 0)\rangle$ as a factor, and by the remark above, this vanishes. 
\end{proof}

\subsection{Bounds for modular symbols}

In this section we prove upper bounds for modular symbols, in similar fashion to \cite[Proposition 3.3]{sullivan} or \cite[Proposition 2.6]{gafa}.

\begin{theorem}
    \label{bound}
    If $ \gamma= \begin{pmatrix} * & * \\ c & * \end{pmatrix} \in \T$, then $\inprod{ \sa \gamma \sigmab^{-1}}{\alpha} \ll |\log |c|| + 1$.
\end{theorem}

\begin{proof}
    We define the antiderivative of $\alpha$:
    \begin{equation}
        \label{antiderivative}
        F_{\a}(P)=\int_{\a}^{P} \alpha \ .
    \end{equation}
    Since $\alpha$ is cuspidal, it follows that it is rapidly decreasing at cusps, and hence $F$ is well-defined on $\h^3 \cup \{ \text{cusps}\}$. We note that
    \begin{equation*}
        F_{\a}(P)=\int_{\a}^P \alpha = \int_{j \infty}^{\sa^{-1} P} \sa^* \alpha .
    \end{equation*}

    We note that $F'_{\a}:= F_{\a} \circ \sa$ is invariant under the translations in $\Lambda_{\a}$. Since $\alpha$ is rapidly decreasing at the cusp $\a$, it follows that $F_{\a}(P)$ is bounded for $y(\sa^{-1 }P)>Y_0$ with $Y_0$ chosen as in \eqref{Y_0}.

Writing $\sa^* \alpha=f_{\a,1} dx_1 + f_{\a,2} dx_2 + f_{\a, 3} dy$, we conclude that
\begin{align*}
    F_{\a}(P) &= \int_{j \infty}^{\sa^{-1} P} f_{\a,1} dx_1 + f_{\a,2} dx_2 + f_{\a, 3} dy \\
    &= \int_{\infty}^{y(\sa^{-1}P) }f_{\a,3}(z, y) dy  \quad (\text{for some } z \in \mathcal{P_{\a}})\\
    &= \int_{\infty}^{Y_0 }f_{\a,3}(z, y) dy  + \int_{Y_0}^{y(\sa^{-1}P)} f_{\a,3}(z, y) dy \\
    &\ll 1 + |\log y(\sa^{-1}P)| .
\end{align*}
We have used the fact that by integrating along a vertical path, we can ignore the contributions from $dx_1$ and $d x_2$, and the last inequality follows from the fact that $f_{\a,3}(z,y) \ll 1/y$, see \eqref{f_i bound}.

We deduce that for $\gamma \in \Gamma$,
\begin{align*}
    \inprod{\gamma}{\alpha} &= F_{\a}(\gamma P) - F_{\a}(P) \\
    & \ll |\log (y( \sa^{-1} \gamma P))| + |\log (y( \sa^{-1 }P))| + 1 
    \ .
\end{align*}

Pick $\gamma  = \begin{pmatrix}  a & b \\ c& d\end{pmatrix} \in \sa^{-1} \Gamma \sigmab$ and $P=\sigmab (0,0,1)$. Then the equation above implies that
\begin{equation*}
    \inprod{ \sa \gamma \sigmab^{-1}}{\alpha} \ll |\log (|c|^2 + |d|^2)| +1 \ .
\end{equation*}

The lower left element $c$ is constant in a double coset in $ \sa^{-1} \Gamma_{\a}' \sa \backslash \sa^{-1} \Gamma \sigmab / \sigmab^{-1} \Gamma_{\b}' \sigmab$ and clearly $|c| \geq c_{\a \b}$. Hence we can choose a representative $\begin{pmatrix}  a & b \\ c& d\end{pmatrix}$ in this double coset such that $|d| \ll |c|$ and we conclude that
\begin{equation*}
    \inprod{\sa \gamma \sigmab^{-1}}{\alpha} \ll | \log |c|| +1 \ . 
\end{equation*}
\end{proof}

\section{Perturbation theory of objects twisted by modular symbols}

In this section we study the dependency on $\epsilon$ of the space $\Ltwo$. If we denote by $\lambda_0(\epsilon)$ the first eigenvalue of $-\Delta$ on $\Ltwo$, we will see that, for $\epsilon$ small enough, $\lambda_0(\epsilon)$ is analytic in $\epsilon$ and we obtain the first few terms in the Taylor expansion around $\epsilon=0$. We also study the behaviour of the residue of $L_{\a \b}(s, \epsilon)$ at $s_0(\epsilon)$, where $s_0(\epsilon)(2-s_0(\epsilon))=\lambda_0(\epsilon)$. 

\subsection{Spectral theory of the space $\Ltwo$}
\label{Laplacian}

Denote by $L^2(\Gamma \backslash \h^3, \chi_\epsilon)$ the space of square integrable functions on $\Gamma \backslash \h^3$ with respect to the hyperbolic metric, satisfying
\begin{equation*}
 f(\gamma P)= \chi_{\epsilon}(\gamma) f(P) \ .
\end{equation*}
For $f,g \in \Ltwo$, we note that $f \overline{g}$ is $\Gamma$-invariant. Hence we define the inner product
\begin{equation*}
    \inprod{f}{g}:= \int_{\GH} f \overline{g} \  d v \ .
\end{equation*}
We let $\mathcal{D}(\epsilon) \subset L^2(\GH, \chi_{\epsilon})$ be the subspace consisting of all $C^2$-functions such that $\Delta f \in L^2(\GH, \chi_{\epsilon})$.
For $f, g \in C^1(\h)$, as in \cite[p. 136]{egm}, we define
\begin{equation}
    \label{gr definition}
    \textbf{Gr}(f,g):=y^2 (f_{x_1} \overline{g_{x_1}}+  f_{x_2} \overline{g_{x_2}} + f_{y} \overline{g_{y}})=[df,dg] \ ,
\end{equation}
where we have used the notation introduced in \eqref{pointwise ip}. Then for all $f,g \in \mathcal{D}(\epsilon)$, $\textbf{Gr}(f,g)$ is $\Gamma$-invariant. Moreover, the following theorem holds, see \cite[Theorem 4.1.7]{egm}.
\begin{theorem}
\label{adjoint}
For all $f,g \in \mathcal{D}(\epsilon)$,
\begin{equation*}
    \int_{\GH} (-\Delta f) \overline{g} d v = \int_{\GH} \mathbf{Gr}(f,g) d v \ .
\end{equation*}
\end{theorem}

In particular, $-\Delta: \mathcal{D}(\epsilon) \to L^2(\Gamma \backslash \h, \chi_{\epsilon}) $ is a symmetric and positive operator. We denote by $\Le$ the closure of $\Delta$ acting on $\mathcal{D}(\epsilon)$.

The theory developed in \cite[Chapter 5]{egm} for $L^2(\GH)$ can be straightforwardly generalised to $L^2(\GH, \chi_{\epsilon})$. The operator $\Le$ is nonnegative, its spectrum consists of a discrete part and a continuous part. Let

\begin{equation*}
    0 \leq \lambda_0(\epsilon) \leq \lambda_1(\epsilon) \leq \cdots \lambda_n(\epsilon)<1
\end{equation*}
be the eigenvalues in the interval $[0,1)$ counted with their multiplicities.

The first eigenvalue is zero if and only if $\epsilon=0$, in which case it is simple and the eigenspace is generated by the constant function. We write $\lambda_n(\epsilon)=s_n(\epsilon)(2-s_n(\epsilon))$, where we choose $1 \leq s_n(\epsilon) \leq 2$ for $0 \leq \lambda_n(\epsilon) \leq 1$.

Recall that since $\alpha$ is cuspidal, there exists some compactly supported 1-form $\widetilde{\alpha}$ such that
\begin{equation*}
    \inprod{\gamma}{\widetilde{\alpha}}=\inprod{\gamma}{\alpha} \quad \text{ for all } \gamma \in \Gamma \ .
\end{equation*}

With this in mind, we define
\begin{equation}
    U_{\a}(P, \epsilon):= \exp \lr{2 \pi i \epsilon \int_{\a}^P \widetilde{\alpha}}
\end{equation}
and consider the unitary operators
\begin{align*}
    U_{\a}(\epsilon): L^2 (\GH) &\to \Ltwo ,\\
    f &\mapsto    U_{\a}(\cdot, \epsilon) f \ .
\end{align*}

We also define 
\begin{equation}
    L(\epsilon):=U_{\a}(\epsilon)^{-1} \Le U_{\a}(\epsilon) \ .
\end{equation}

This implies that $L(\epsilon)=\Delta$ outside the support of $\widetilde{\alpha}$. This will be crucial later in the paper, particularly in the proof of Lemma 4.7.

This construction ensures that the operator $L(\epsilon)$ acts on the fixed space $L^2(\GH)$ and that $L(\epsilon)$ and $\Le$ are unitary equivalent. This implies that $\hbox{Spec}(L(\epsilon))=\hbox{Spec}(\Le)$.

Write $\widetilde{\alpha}= f_1 dx_1 + f_2 dx_2 + f_3 dy$. Using the fact that 
\begin{equation*}
    \pdv{U_{\a}(P, \epsilon)}{x_1}= 2 \pi i \epsilon f_1(P) U_{\a}(P, \epsilon)
\end{equation*}
and the other two similar corresponding derivatives with respect to $x_2$ and $y$, we observe that
\begin{align*}
    L(\epsilon)h =& U_{\a}(P, \epsilon)^{-1} \lr{ y^2 \lr{\pdv[2]{}{x_1} + \pdv[2]{}{x_2} + \pdv[2]{y}} - y \pdv{y} } \lr{U_{\a}(P, \epsilon) h} \\
    =& \Delta h + 4 \pi i  \epsilon y^2\lr{f_1 \pdv{h}{x_1}+ f_2 \pdv{h}{x_2}+ f_3 \pdv{h}{y}} + 2 \pi i \epsilon y^2 \lr{\pdv{f_1}{x_1}+ \pdv{f_2}{x_2}+ \pdv{f_3}{y}}h \\
    &- 4 \pi^2 \epsilon ^2 y^2  (f_1^2 + f_2^2 + f_3^2) h  - 2 \pi i\epsilon y f_3 h .
\end{align*}

We conclude that 
\begin{equation}
\label{L(e) equation}
    L(e) h =  \Delta h + \epsilon L^{(1)} h + \epsilon^2 L^{(2)}h, 
\end{equation}where
\begin{align*}
    L^{(1)}h &= 2 \pi i  \lr{y^2 \lr{\pdv{f_1}{x_1}+ \pdv{f_2}{x_2} + \pdv{f_3}{y}} - y f_3} + 4 \pi i [dh, \alpha] \ , \\ 
    L^{(2)}h &= - 4 \pi^2 \epsilon ^2 [\alpha,\alpha] h \ .
\end{align*}
In particular we note that $L(\epsilon)$ is independent of the choice of the cusp $\a$ and as a function of $\epsilon$ is a polynomial of degree two.

From now on, we fix $Y_0$ in \eqref{Y_0} large enough such that $\widetilde{\alpha}$ vanishes on cuspidal sectors $\mathcal{F}_{\a}(Y_0)$, for all cusps $\a$. Fix $Y > Y_0$. We choose $h \in C^{\infty}(\r^+)$ such that $h(y)=0$ for $y \leq Y$ and $h(y)=1$ for $y \geq Y+1$. Then for $s \in \c$ and $P \in \mathcal{F}$ we define
\begin{equation*}
    h_{\a}(P,s) := \begin{cases} h(y(\sa^{-1}P)) y(\sa^{-1}P)^s & \text{if } P \in  \mathcal{F}_{\a}(Y_0), \\ 0 & \text {if } P \in \mathcal{F} \setminus  \mathcal{F}_{\a}(Y_0).\end{cases}
\end{equation*}

We extend $h_\a(\cdot, s)$ to a $\Gamma$-invariant $C^\infty$-function defined for $s \in \c$ and $P \in \h^3$.

We also define
\begin{equation}
    \label{Omega(epsilon)}
    \Omega_{\epsilon}=\left \{ s \in \c \ | \ \Re(s)>1, \ s(2-s) \not \in \mbox{Spec}(-L(\epsilon))\right \} \ .
\end{equation}

We have the following Lemma, similar to \cite[Lemma 6.1.4]{egm}, \cite[Lemma 2.1]{spectral_deformations} or \cite[Lemma 3.1]{petridis}:

\begin{lemma}
    \label{analytic_in_epsilon}
    For $s \in \Omega_{\epsilon}$, there exists a unique $D_a(P, s, \epsilon)$ such that
    \begin{equation}
        \label{unique}
        (L(\epsilon)+s(2-s))D_{\a}(P,s, \epsilon)=0, \quad D_{\a}( P,s,\epsilon)- h_{\a}(P,s) \in L^2 (\GH) \ .
    \end{equation}
    Moreover, $D_{\a}(P,s,\epsilon)$ is holomorphic in $s \in \Omega_{\epsilon}$ and real analytic in $\epsilon$.
\end{lemma}

\begin{proof}
    If such a solution exists, we write $$g_{\a}(P, s, \epsilon) = D_{\a}(P,s,\epsilon)-h_{\a}(P,s) \in L^2(\GH) \ .$$ We apply $(L(\epsilon)+ s(2-s))$ to deduce
    \begin{equation}
        (L(\epsilon)+s(2-s))g_{\a}(P,s,\epsilon)=H_{\a}(P,s, \epsilon) \ ,
    \end{equation}
    where
    \begin{equation}
        \label{H(P,s))}
        H_{\a}(P,s, \epsilon)=-(L(\epsilon)+s(2-s))h_{\a}(P,s) \ .
    \end{equation}
    We note that $H_{\a}$ is a $\Gamma$-invariant $C^{\infty}$-function in the variable $P$, which is moreover of compact support when restricted to $\mathcal{F}$. It also depends holomorphically on $s \in \Omega_{\epsilon}$. Moreover, since $L(\epsilon)$ is equal to $\Delta$ outside the support of $\widetilde{\alpha}$, we observe that $H_{\a}$ is independent from $\epsilon$, so that we can write it as $H_{\a}(P,s)$.
    
    We can now use \eqref{H(P,s))} as a definition for $H_{\a}(P,s)$, and for $s \in \Omega_{\epsilon}$, we can apply the resolvent operator defined as
    \begin{equation*}
        R(s, \epsilon) = (L(\epsilon)+s(2-s))^{-1}
    \end{equation*}
    to obtain a unique function
    \begin{equation*}
        g_{\a}(P,s, \epsilon) =   R(s, \epsilon) H_{\a}(P,s) \in L^2(\GH) \ .
    \end{equation*}
    Since there exist only finitely many values of $s \in \c$ with $\Re(s)>1$ for which $s(2-s)$ is an eigenvalue of $-\Delta=-L(0)$ and we know that $L(\epsilon)$ is a polynomial in $\epsilon$ given by \eqref{L(e) equation}, we can use the arguments in \cite[p. 66--67]{kato} to conclude that the resolvent $R(s, \epsilon)$ is holomorphic for $s\in \Omega_{\epsilon}$ and depends real analytically on $\epsilon$. 
    \end{proof}

\begin{remark}
    \label{eisenstein remark}
    For $\Re(s)>2$, the equation \eqref{unique} agrees with 
    \begin{equation*}
        D_{\a}(P,s, \epsilon) =   U_{\a}(\epsilon)^{-1} E_{\a} (P, s , \epsilon) \ .
    \end{equation*}
    Therefore, the conclusions of Lemma \ref{analytic_in_epsilon} hold for the Eisenstein series in the region $s \in \Omega_{\epsilon}$. \\
\end{remark}

\subsection{Behavior of $\lambda_0(\epsilon)$ and the residue of $L_{\a\b}(s, \epsilon)$ at $s_0(\epsilon)$}

We know that $\lambda_0(0)=0$ is a simple eigenvalue for $L(0)=\Delta$. It is possible to apply Kato's perturbation theory for finite dimensional spaces \cite[p. 68--70]{kato} for our operator $L(\epsilon)$ of the form \eqref{L(e) equation}, as explained in \cite[Section 4]{fermat_curves}. We conclude that for $\epsilon$ in a small interval around 0, $\lambda_0(\epsilon)$ is real analytic in $\epsilon$ and also $\lambda_0(\epsilon)$ is a simple eigenvalue.

We let $u_0(P, \epsilon) \in \Ltwo$ be the normalised corresponding eigenfunction of $-\widetilde{L}(\epsilon)$, i.e.
\begin{equation}
    \label{u_0}
    -\Le u_0(P, \epsilon)= \lambda_0(\epsilon) u_0 (P, \epsilon) \quad \text{and} \quad  \int_{\GH} |u_0(P, \epsilon)|^2 d v =1 \ .
\end{equation}
We want to study the behaviour of $\lambda_0(\epsilon)$ around $\epsilon=0$. We adapt the proof of \cite[Lemma 2.1]{phillips}.

\begin{lemma}
\label{lemmal0}
We have that $\lambda_0'(0)=0$ and 
\begin{equation*}
    \lambda_0'' (0)= \frac{8 \pi^2}{\vol} \| \alpha \|_2^2 \ .
\end{equation*}
\end{lemma}

\begin{proof}
We apply Theorem \ref{adjoint} with $f(P)=g(P)=u_0(P, \epsilon)$ to obtain
\begin{equation}
    \label{lambda_0}
    \lambda_0(\epsilon)= \int_{\GH} \textbf{Gr}(u_0(\cdot, \epsilon), u_0(\cdot, \epsilon)) d v =\int_{\GH} y^2  \left (\left | \pdv{u_0}{x_1}\right |^2 + \left | \pdv{u_0}{x_2}\right |^2 + \left | \pdv{u_0}{y}\right |^2\right )\frac{dx_1 dx_2 dy}{y^3} .
\end{equation}
In particular, we note that $\lambda_0(\epsilon) \geq 0$ and $\lambda_0(\epsilon)=0$ if and only if the $u_0(P,\epsilon)$ is constant iff $\epsilon=0$.

We differentiate \eqref{lambda_0} with respect to $\epsilon$, yielding
\begin{equation}
    \lambda_0'(\epsilon) = 2 \int_{\GH} \textbf{Gr} \lr{\pdv{u_0}{\epsilon} (\cdot, \epsilon), u_0(\cdot, \epsilon)} dv \ .
\end{equation}
Setting $\epsilon=0$ we deduce that $\lambda_0'(0)=0$ since $u_0(P,0)$ is a constant function. Differentiating once again, 
\begin{equation}
    \label{l0"}
    \lambda_0''(\epsilon) = 2 \int_{\GH} \lr{ \textbf{Gr} \lr{\pdv[2]{u_0}{\epsilon} (\cdot, \epsilon), u_0(\cdot, \epsilon)} + \textbf{Gr} \lr{\pdv{u_0}{\epsilon} (\cdot, \epsilon), \pdv{u_0}{\epsilon} (\cdot, \epsilon)}} d v \ .
\end{equation}
We define
\begin{equation*}
    w(P):=\eval{\pdv{u_0}{\epsilon} (P,0)}_{\epsilon=0} \ .
\end{equation*}
Hence \eqref{l0"} and \eqref{in_prod_1_forms} give us 
\begin{equation}
    \lambda_0''(0)= 2 \int_{\GH} \textbf{Gr}(w,w) d v  = 2 \|dw \|_2^2. \label{dvdv}
\end{equation}
since the mixed term vanished because $u_0(\cdot, 0)$ is constant. 

Since $u_0 (P, \epsilon) \in \Ltwo$, we know that $u_0(\gamma P, \epsilon)= \chi_{\epsilon}(\gamma) u_0 (P, \epsilon)$. Differentiating this equation with respect to $\epsilon$ and then setting $\epsilon=0$, we obtain that for all $\gamma \in \Gamma$,
\begin{equation}
    w(\gamma P)= w(P) + \frac{2 \pi i \inprod{\gamma}{\alpha}}{\sqrt{\vol}} \ ,
\end{equation}
where we have used the fact that $u_0(P, 0) = 1/ \sqrt{\vol}$. Moreover, since we know that $\lambda_0(0)=\lambda_0'(0)=0$, we know from \eqref{u_0} that
\begin{equation}
    \Delta w =0 \ .
\end{equation}

If we define $\beta= dw- 2 \pi i \vol^{-1/2} \alpha$, then $\beta$ is a harmonic, $\Gamma$-invariant 1-form such that for all $P \in \h^3$ and $\gamma \in \Gamma$
\begin{equation}
    \int_P^{\gamma P} \beta = 0 \ .
\end{equation}
In other words, this means that $\inprod{\gamma}{\beta}=0$, for all $\gamma \in \Gamma$, and since we have a perfect pairing, this implies that $dw$ and $2 \pi i \vol^{-1/2} \alpha$ are in the same cohomology class. The result then follows from \eqref{dvdv}. 
\end{proof}

\begin{remark}
\label{remark_derivative}
From the proof above, we can deduce that $w$ is of the form
\begin{equation*}
    w(P)=\frac{2 \pi i}{\sqrt{\vol}}\int_Q^P \alpha + C_Q   \quad \textrm{for some } Q \in \h^*,
\end{equation*}
where $C_Q$ is a constant.
\end{remark}

\begin{corollary}
    Let $$C_{\alpha}= \frac{4  \| \alpha \|_2^2}{\vol} \ . $$ Then 
    \begin{equation*}
        s_0(\epsilon)=2-\pi^2 C_{\alpha} \epsilon^2 + O(\epsilon^3) \ .
    \end{equation*}
\end{corollary}
\begin{proof}
    It follows immediately from Lemma \ref{lemmal0} and the fact that $\lambda_0(\epsilon)=s_0(\epsilon)(2-s_0(\epsilon))$.  \\
\end{proof}

\begin{lemma} We have that
    \label{residue at s_0}
    \begin{equation*}
        \Res_{s=s_0(\epsilon)} L_{\a \b}(s, \epsilon)=  \frac{| \mathcal{P}_{\a}|| \mathcal{P}_{\b}| [\Gamma_{\a}: \Gamma_{\a}'] }{ \pi \vol } + \epsilon   \frac{2  i |\mathcal{P}_{\a}|  |\mathcal{P}_{\b}| [\Gamma_{\a}: \Gamma_{\a}']}{\vol} \int_{\a}^{\b} \alpha + O(\epsilon^2) \ .
    \end{equation*}
\end{lemma}
\begin{proof}
From the Fourier expansion \eqref{fourier expansion} of the Eisenstein series, we deduce
\begin{equation*}
    \int_{\mathcal{P}_{\b}} E_{\a}(\sigmab P, s, \epsilon) dx_1 dx_2 = \delta_{\a \b}|\mathcal{P}_{\b}| [\Gamma_{\a}: \Gamma_{\a}']y^s + \frac{\pi}{s-1}L_{\a \b}(s , \epsilon)y^{2-s} .
\end{equation*}

We look at the residue at $s=s_0(\epsilon)$ on both sides of the equality to obtain
\begin{align*}
    \frac{\pi y^{2-s_0(\epsilon)}}{s_0(\epsilon)-1}\Res_{s=s_0(\epsilon)}L_{\a \b}(s, \epsilon)&=\Res_{s=s_0(\epsilon)} \int_{\mathcal{P}_{\b}} E_{\a}(\sigmab P, s, \epsilon) dx_1 dx_2\\
    &= \int_{\mathcal{P}_{\b}} u_{\a}(\sigmab P, \epsilon) dx_1 dx_2 \ ,
\end{align*}
where
\begin{equation*}
    u_{\a}(P, \epsilon):= \Res_{s=s_0(\epsilon)}E_{\a}(P, s, \epsilon) \ .
\end{equation*}

Since $s_0(\epsilon)=2+ O(\epsilon^2)$, it follows that
\begin{equation*}
    \eval{\pdv{\lr{\Res_{s=s_0(\epsilon)}L_{\a \b}(s, \epsilon)}}{\epsilon}}_{\epsilon=0} = \frac{1}{\pi}\int_{\mathcal{P}_{\b}} v_{\a}(\sigmab P) dx_1 dx_2 \ ,
\end{equation*}
where
\begin{equation*}
    v_{\a}(P):=\eval{\pdv{u_{\a}(P, \epsilon)}{\epsilon}}_{\epsilon=0}  .
\end{equation*}

We define 
\begin{equation}
\label{wa}
    w_{\a}(P, \epsilon)=U_{\a}(P, \epsilon)^{-1} u_{\a}(P, \epsilon).
\end{equation} 
Then $w_{\a}(P, \epsilon) \in L^2(\GH)$ and it is an eigenfunction of $L(\epsilon)$ with eigenvalue $\lambda_0(\epsilon)$. Differentiating \eqref{wa} with respect to $\epsilon$ and then setting $\epsilon=0$, we get
\begin{equation*}
    v_{\a}(P)=2 \pi i \  w_{\a}(P,0) \int_{\a}^P \widetilde{\alpha}  + \eval{\pdv{w_{\a}(P, \epsilon)}{\epsilon}}_{\epsilon=0} \ .
\end{equation*}

We note that when  $P$ is in the cuspidal sector $\mathcal{F}_{\a}(Y_0)$, the right-hand side of the equality above is zero, since $\widetilde{\alpha}$ is compactly supported, and $L(\epsilon)=\Delta$ in this region, hence $w_{\a}(\cdot, \epsilon)$ is constant in this region.

Now, as in \eqref{eisenstein residue}, we know that $u_{\a}(P,0)= \frac{|\mathcal{P}_{\a}| [\Gamma_{\a}: \Gamma_{\a}']}{\vol}$, so by using the results in Lemma \ref{lemmal0} and Remark \ref{remark_derivative}, we deduce that 
\begin{equation*}
    v_{\a}(P)=\frac{2 \pi i |\mathcal{P}_{\a}| [\Gamma_{\a}: \Gamma_{\a}']}{\vol} \int_{\a}^P \alpha \ .
\end{equation*}

Since $\alpha$ is a cuspidal one-form, from definition we know that
\begin{equation*}
    \int_{\mathcal{P}_{\b}} \lr{\int_{\b}^{\sigmab P} \alpha} dx_1 dx_2  = \int_{\mathcal{P}_\b} \lr{ \int_{j \infty}^P \sigmab^* \alpha} dx_1 dx_2=0 \ ,
\end{equation*}
hence
\begin{align*}
    \eval{\pdv{\lr{\Res_{s=s_0(\epsilon)}L_{\a \b}(s, \epsilon)}}{\epsilon}}_{\epsilon=0} = \frac{2  i |\mathcal{P}_{\a}| [\Gamma_{\a}: \Gamma_{\a}']}{\vol} \int_{\mathcal{P}_{\b}}\int_{\a}^{\sigmab P} \alpha = \frac{2  i |\mathcal{P}_{\a}|  |\mathcal{P}_{\b}| [\Gamma_{\a}: \Gamma_{\a}']}{\vol} \int_{\a}^{\b} \alpha \ . 
\end{align*}
\end{proof}

\section{Moment generating function}

In this section we study the exponential sum
\begin{align}
\begin{split}
    \sum_{\gamma \in \T(X)} \chi_{\epsilon}(\sa \gamma \sigmab^{-1}) &= \sum_{\gamma \in \T(X)} \exp(2 \pi i \epsilon \inprod{\sa \gamma \sigmab^{-1}}{\alpha})\\ &= [\Gamma_{\b} : \Gamma_{\b}'] \sum_{r \in R_{\a \b}(X)} \exp(2 \pi i \epsilon \rr_{\a \b})
\end{split}
\end{align}
which is the moment generating function for the distribution of modular symbols. We relate this sum to the generating series $L_{\a \b}(s, \epsilon)$. We write the first few terms in the Taylor expansion around $\epsilon=0$, thus obtaining expressions for the first and second moments of the distribution of modular symbols. Additionally, we show that the values in the set $R_{\a \b}(X)$ become equidistributed modulo the lattice $\Lambda_{\a}$ as $X \to \infty$.

Firstly, we need the following lemma about bounds on vertical lines for $L_{\a \b}(s, 0, \mu, \epsilon)$.
\begin{lemma}
\label{vertical lines}
  Fix some $\delta >0$. If $1 + \delta < \Re(s) < 2 + \delta$ and $s(2-s)$ bounded away from spectrum of $L(\epsilon)$, then, uniformly in $\epsilon$,
 \begin{equation}
     L_{\a \b}(s, 0, \mu, \epsilon) \ll_{\delta} (1+ |\mu|)^{2-\Re(s)+ \delta} |s| \ .
 \end{equation}
\end{lemma}

\begin{proof}

For $\mu =0$, we use a similar argument of that in \cite[p.~655]{selberg} (which follows from the Maa{\ss}--Selberg relations in $\h^3$ \cite[p. 110]{egm}). We have that $|\phi_{\a \b}(s, \epsilon)|=O(1)$ in the region $\Re(s)>1+\eta$ and away from the spectrum of $L(\epsilon)$. Now, the result follows from \eqref{phi definition}.

When $\mu \neq 0$ and $\Re(s)>1$, we use Lemma \ref{integral rep}. Choose $w=2+ 2 \delta + it$, where $s= \sigma + it$. Then Stirling's formula gives us that the contribution from the Gamma factors is $O(|s|)$.

Next, we want to study the contribution from the integral. We use Lemma \ref{analytic_in_epsilon} to deduce that for $\Re(s)>1$ and $s(2-s)$ bounded away from the spectrum of $L(\epsilon)$,
\begin{align*}
    \int_{\mathcal{F}} |E_{\a}(P, s, \epsilon) & \overline{E_{\b, \mu}(P, \overline{w}, \epsilon)}| d v 
    =  \int_{\mathcal{F}} |D_{\a}(P, s, \epsilon) E_{\b, \mu}(P, \overline{w}, \epsilon)| d v \\
    &\leq  \int_{\mathcal{F}} | h_{\a}(P,s) E_{\b, \mu}(P, \overline{w}, \epsilon)| d v + \int_{\mathcal{F}} |(D_{\a}(P, s, \epsilon) - h_{\a}(P,s)) E_{\b, \mu}(P, \overline{w}, \epsilon)| d v \ .
\end{align*}
The second integral is bounded by
\begin{equation*}
    \| g(\sa^{-1}P, s, \epsilon) \|_{L^2} \| E_{\b, \mu}(P, \overline{w}, \epsilon) \|_{L^2} \ll 1 \ .
\end{equation*}
It remains to study the first integral. It suffices to concentrate on the cuspidal sector $\mathcal{F}_{\a}(Y)$ since $h_{\a}(P,s)$ vanishes everywhere else. We get
\begin{align*}
    \int_{\mathcal{F}_{\a}(Y)} | h_{\a}(P,s)E_{\b, \mu}(P, \overline{w}, \epsilon)| d v = \int_{Y}^{\infty} \int_{\mathcal{P}_{\a}} | y^s E_{\b, \mu}( \sa P, \overline{w}, \epsilon)| \ d v \ .
\end{align*}

Now, with our choice of $w=2+ 2 \delta + it$, we see that $E_{\b, \mu}( \sa P, \overline{w}, \epsilon)$ decays exponentially in the cusp, so the integral above is indeed bounded. This in turn implies that 
\begin{equation*}
    \int_{\mathcal{F}} |E_{\a}(P, s, \epsilon) \overline{E_{\b, \mu}(P, \overline{w}, \epsilon)}| d v \ll 1 \ ,
\end{equation*}
and hence we obtain the desired upper bound for $L_{\a \b}(0, \mu, s, \epsilon)$. 
\end{proof}

We obtain the following expression for the moment generating function by using a similar method to \cite[Section 4]{gafa}.
\begin{lemma}
    \label{main}
    There exists an absolute constant $\nu>0$ depending on the spectral gap of $\Delta$ such that, uniformly for $\epsilon$ small enough, 
    \begin{equation*}
        \sum_{\gamma \in \T(X)} \overline{\chi}_{\epsilon}(\sa \gamma \sigmab^{-1}) = \frac{X^{2s_0(\epsilon)}}{s_0(\epsilon)} \Res_{s=s_0(\epsilon)} L_{\a \b}(s, \epsilon)\lr{1+O(X^{-\nu})} \ .
    \end{equation*}
\end{lemma}

\begin{proof}
    Let $\phi_U: \r \to \r$ be a family of smooth nonincreasing functions with
    \begin{equation}
        \label{phi lemma}
        \phi_U(t)= \begin{cases} 1 \quad \text{if } t\leq 1-1/U, \\ 0 \quad \text{if } t  \geq 1+ 1/U,
        \end{cases}
    \end{equation}
    and $\phi_U^{(j)}(t) = O(U^j)$ as $U \to \infty$. For $\Re (s)>0$, we consider the Mellin transform
    \begin{equation}
        R_U(s) = \int_{0}^{\infty} \phi_{U}(t) t^{s} \frac{dt}{t} \ .
    \end{equation}
    We can easily see that
    \begin{equation}
        \label{approx1}
        R_U(s)=\frac{1}{s}+ O \lr{\frac{1}{U}} \quad \text{as } U \to \infty
    \end{equation}
    and for any $c>0$
    \begin{equation}
        \label{growth}
        R_U(s)=O \lr{\frac{1}{|s|}\lr{\frac{U}{1+|s|}}^c} \quad \text{as } |s| \to \infty \ ,
    \end{equation}
    where the last estimate follows from repeated partial integration. Now we use the Mellin inversion to obtain
    \begin{align*}
        \sum_{\gamma \in \T}  \overline{\chi}_{\epsilon}(\sa \gamma \sigmab^{-1}) \  \phi_U \lr{\frac{|c|^2}{X^2}} &=  \sum_{\gamma \in \T}  \overline{\chi}_{\epsilon}(\sa \gamma \sigmab^{-1}) \frac{1}{2 \pi i}\int_{\Re(s)=3} \frac{X^{2s}}{|c|^{2s}} R_U(s)   ds\\ &= \frac{1}{2 \pi i} \int_{\Re(s)=3} L_{\a \b}(s, \epsilon) X^{2s} R_U(s) ds \ .
    \end{align*}
    Next, we recall Lemma \ref{vertical lines} and equation \eqref{growth} to deduce that the last integral is absolutely convergent. We want to move the line of integration to $\Re(s)=h$, where $$h=\frac{2 \max (s_1(0), 1)+2}{3}.$$
    Then for $\epsilon$ small enough, $s_1(\epsilon) <h < s_0(\epsilon)$. We integrate along a box of height $T$ and let $T \to \infty$. Indeed, the polynomial growth on vertical lines of $L_{\a \b}(s, \epsilon)$ guaranteed by Lemma \ref{vertical lines}, together with equation \eqref{growth}, give us
    \begin{align*}
        \lim_{T \to \infty} \int_{\substack{\Re(s)=3 \\ |t| \geq T}}  L_{\a \b}(s, \epsilon) X^{2s} R_U(s) ds =  \lim_{T \to \infty} \int_{\substack{\Re(s)=h \\ |t| \geq T}}  L_{\a \b}(s, \epsilon) X^{2s} R_U(s) ds = 0 \ ,
    \end{align*}
    and
    \begin{equation*}
        \lim_{T \to \infty} \int_{\substack{h \leq \Re(s) <2 \\ \Im (s) = T}} L_{\a \b}(s, \epsilon) X^{2s} R_U(s) ds = \lim_{T \to \infty} \int_{\substack{h \leq \Re(s) <2 \\ \Im (s) = - T}} L_{\a \b}(s, \epsilon) X^{2s} R_U(s) ds = 0 \ .
    \end{equation*}
    We conclude that
    \begin{align*}
    \begin{split}
        \frac{1}{2 \pi i} \int_{\Re(s)=3} L_{\a \b}(s, \epsilon) X^{2s} R_U(s) ds =& \frac{1}{2 \pi i} \int_{\Re(s)=h} L_{\a \b}(s, \epsilon) X^{2s} R_U(s) ds \\ &+ \Res_{s=s_0(\epsilon)} \lr{L_{\a \b}(s, \epsilon) X^{2s} R_U(s)} \ .
    \end{split}
    \end{align*}
    Setting $c=3$ in \eqref{growth}, we observe that
    \begin{equation*} 
         \int_{\Re(s)=h} L_{\a \b}(s, \epsilon) X^{2s} R_U(s) ds \ll X^{2h} U^3 \ .
    \end{equation*}
    Now, \eqref{approx1} gives us
    \begin{equation}
        \label{residue at s0}
        \Res_{s=s_0(\epsilon)}\lr{L_{\a \b}(s, \epsilon) X^{2s} R_U(s)} = \frac{X^{2s_0(\epsilon)}}{s_0(\epsilon)}\lr{\Res_{s=s_0(\epsilon)} L_{\a \b}(s, \epsilon)+ O \lr{\frac{1}{U}}} \ .
    \end{equation}
    
    Since we want this to be the main contribution, we choose $U=X^a$, where 
    $$ a= \frac{2-\max (s_1(0), 1)}{4}.$$ With this choice, for $\epsilon$ small enough, we get
    \begin{equation}
        \label{almost there}
         \sum_{\gamma \in \T}  \overline{\chi}_{\epsilon}(\sa \gamma \sigmab^{-1}) \  \phi_U \lr{\frac{|c|^2}{X^2}} = \frac{X^{2s_0(\epsilon)}}{s_0(\epsilon)}\lr{\Res_{s=s_0(\epsilon)} L_{\a \b}(s, \epsilon)+ O(X^{-a})} \ .
    \end{equation}
    
    Setting $\epsilon=0$, using Lemma \ref{residues}, we obtain
    \begin{equation*}
        \sum_{ \gamma \in \T} \phi_U\lr{\frac{|c|^2}{X^2}} = X^4 \lr{\frac{| \mathcal{P}_{\a}|| \mathcal{P}_{\b}| [\Gamma_{\a}: \Gamma_{\a}'] }{ 2 \pi \vol } + O(X^{-a})} \ .
    \end{equation*}
    
     We now choose $\phi_U^1$ and $\phi_U^2$ as in \eqref{phi lemma} with the further requirements that $\phi_U^1(t)=0$ for $t \geq 1$ and  $\phi_U^2(t)=1$ for $0 \leq t \leq 1$. Then
    \begin{equation*}
        \sum_{\gamma \in \T}    \phi_U^1 \lr{\frac{|c|^2}{X^2}} \leq \sum_{\gamma \in \T(X)} 1 \leq \sum_{\gamma \in \T}    \phi_U^2 \lr{\frac{|c|^2}{X^2}} \ ,
    \end{equation*}
    so the previous two equations give us 
    \begin{equation}
        \label{size o T(X)}
            \# \T(X) =  X^4 \lr{\frac{| \mathcal{P}_{\a}|| \mathcal{P}_{\b}| [\Gamma_{\a}: \Gamma_{\a}'] }{ 2 \pi \vol } + O(X^{-a})} \ .
    \end{equation}
    
    Also, from the definition of $\phi_U$,
    \begin{align}
    \label{there}
       \sum_{\gamma \in \T}  \overline{\chi}_{\epsilon}(\sa \gamma \sigmab^{-1}) \  \phi_U \lr{\frac{|c|^2}{X^2}} = \sum_{\gamma \in \T(X)}  \overline{\chi}_{\epsilon}(\sa \gamma \sigmab^{-1}) +  O \lr{ \# \left \{ \gamma \in \T \ : \ 1-\frac{1}{U} \leq \frac{|c|^2}{X^2} \leq 1+ \frac{1}{U}\right \}} \ .
    \end{align}
   But now we use \eqref{size o T(X)} to bound the size of the error term
   \begin{align*}
       \# \left \{ \gamma \in \T \ : \ 1-\frac{1}{U} \leq \frac{|c|^2}{X^2} \leq 1+ \frac{1}{U}\right \} = \T \lr{X \sqrt{1+\frac{1}{U}}} -  \T \lr{X \sqrt{1-\frac{1}{U}}} =  O \lr{X^{4-a/2}} \ .
   \end{align*}
    The conclusion follows from \eqref{almost there} and \eqref{there}. 
\end{proof}

Let 
\begin{equation*}
    F(\epsilon)=\Res_{s=s_0(\epsilon)}L_{\a \b}(s, \epsilon) 
\end{equation*}
and we write its Taylor expansion around $\epsilon=0$ as $F(\epsilon)=\sum_{k \geq 0} C_k \epsilon^k$. So far we have shown that
\begin{equation*}
    C_0=  \frac{| \mathcal{P}_{\a}|| \mathcal{P}_{\b}| [\Gamma_{\a}: \Gamma_{\a}'] }{ \pi \vol } \quad \textrm{and} \quad C_1= \frac{2  i |\mathcal{P}_{\a}|  |\mathcal{P}_{\b}| [\Gamma_{\a}: \Gamma_{\a}']}{\vol} \int_{\a}^{\b} \alpha \ .
\end{equation*}

We note that the coefficients $C_k$ were essentially computed by Petridis--Risager in \cite{petridis}, allowing them  to obtain all moments for modular symbols.

\begin{corollary}
\label{moments}
If $ \epsilon \geq X^{-\nu/4}$, for some $\nu>0$ depending on the spectral gap, then
    \begin{equation*}
        \frac{1}{\#\T(X)}\sum_{\gamma \in \T(X)} \chi_{\epsilon}(\sa \gamma \sigmab^{-1})=1+ \epsilon \lr{2 \pi i \int_{\b}^{\a} \alpha} +  \epsilon^2 \lr{-2 \pi^2 \log X C_{\alpha}+D_{\alpha, \a \b}}   +O(X^{-\nu}) \ ,
    \end{equation*}
    where
    \begin{equation}
        \label{Dalpha}
        D_{\alpha, \a \b} = 2 \pi ^2 C_{\alpha} + \frac{C_2}{C_0} \ .
    \end{equation}
\end{corollary}

\begin{remark}
From the formula above we observe that computing the variance shift $D_{\alpha, \a \b}$ is equivalent to finding the second term in Laurent series expansion of $\eval{\pdv[2]{\epsilon} L_{\a\b} (s, \epsilon)}_{\epsilon=0}$, or in other words finding the first two terms in the Laurent expansion of the Goldfeld Eisenstein series $E_{\a}^2(P,s)$. For the case of $\h^2$ this is done in \cite{petridis} and their methods could be extended to work in $\h^3$ as well.
\end{remark}

As a consequence of our work so far, we can show that $R_{\a \b}(X)$ is equidistributed in the fundamental domain $\mathcal{P}_{\a}$ as $X \to \infty$. 

\begin{proposition}
 There exists $\nu>0$ depending on the spectral gap for $\Delta$, such that for all $\mu \in \Lambda_{\a}^{\circ}$,
\begin{equation*}
    \sum_{r \in R_{\a \b}(X)} e(\inprod{\mu}{r})= \delta_0(\mu) \frac{|\mathcal{P}_{\a}||\mathcal{P}_{\b}| [\Gamma_{\a}: \Gamma_{\a}']}{2 \pi \vol [\Gamma_{\b}: \Gamma_{\b}'] } X^4 + O \lr{(1 + |\mu|) X^{4- \nu}} \ .
\end{equation*}
In particular, for any continuous function $h: \c / \Lambda_{\a} \to \c$, 
\begin{equation*}
    \frac{\sum_{r \in R_{\a \b}(X)}h(r)}{\# R_{\a \b}(X)} \to \int_{\c / \Lambda_{\a}} h(z) dz  \quad \text{as} \quad X \to \infty \ .
\end{equation*}
\end{proposition}

\begin{proof}
From Lemma \ref{T to R}, the generating series for the exponential sum is
\begin{equation*}
   \sum_{r \in R_{\a \b}(X)} \frac{e(\inprod{\mu}{r})}{|c|^{2s}} = \frac{1}{[\Gamma_{\b} : \Gamma_{\b}']} \sum_{\gamma \in \T} \frac{ e \lr{\inprod{\mu}{\gamma \infty}}}{|c|^{2s}} = \frac{1}{[\Gamma_{\b} : \Gamma_{\b}']}  L_{\a \b}(s, \mu, 0, 0) \ .
\end{equation*} 
By inverting $\gamma$ in the series above, we note that $L_{\a \b}(s, \mu, 0, , 0)=L_{\b \a}(s, 0, -\mu, 0)$. We use a contour integration argument similar to the one in the proof of Lemma \ref{main}. The polynomial growth of  $L_{ \b \a} (s, 0, -\mu, 0)$ on vertical lines is guaranteed by Lemma \ref{vertical lines}, whilst by Lemma \ref{residues} we know that $L_{\b \a} (s , 0, \mu, 0)$ has a pole at $s=2$ if and only if $\mu = 0$. Finally, from \eqref{size o T(X)} we know that
\begin{equation*}
    \# R_{\a \b}(X) =  X^4 \lr{\frac{| \mathcal{P}_{\a}|| \mathcal{P}_{\b}| [\Gamma_{\a}: \Gamma_{\a}'] }{ 2 \pi \vol [\Gamma_{\b} : \Gamma_{\b}'] } + O(X^{-\nu})} \ .
\end{equation*}

The second claim follows from the generalised Weyl equidistribution criterion.
\end{proof}

\section{Normal distribution of modular symbols}

We now have all the ingredients to prove that modular symbols have asymptotically a normal distribution. We make use of the Berry--Esseen inequality and of our results about the behaviour of $s_0(\epsilon)$ and $L_{\a \b}(s, \epsilon)$.

We recall the Berry--Esseen inequality, see \cite[Theorem II.7.16]{tenenbaum}.
\begin{theorem}
    \label{be}
    If $X$ is a real valued random variable and $T>0$, then
    \begin{equation}
        \label{BE equation}
        \sup_{z \in \r} \left | \int_{-\infty}^z e^{-t^2/2} dt - \mathbb{P}(X<z) \right | \ll \frac{1}{T}+ \int_{-T}^T \left | \frac{e^{-t^2/2}-\e(\exp(i t X ))}{t}\right | dt \ .
    \end{equation}
\end{theorem}

For $\gamma \in \T(X)$, we define the random variable 
\begin{equation}
    A_\gamma =  \sqrt{\frac{1}{ C_{\alpha} \log X}}  \inprod{\sa \gamma \sigmab^{-1}} {\alpha}
\end{equation}
where $\gamma$ is chosen uniformly at random from $\T(X)$.

We fix $t:= 2 \pi \epsilon \sqrt{ C_{\alpha} \log X}$. Then, by definition,
\begin{equation*}
    \e(\exp(i t A_{\gamma})) = \frac{1}{\# \T(X)} \sum_{\gamma \in \T(X)}  \chi_{\epsilon}(\sa \gamma \sigmab^{-1}) \ .
\end{equation*}

Fix some $\delta>0$. We choose $T=(\log X)^{1/2-\delta}$ and apply Theorem \ref{be} for the random variables $A_{\gamma}$. We split the integral on the right-hand side of \eqref{BE equation} into three ranges, depending on the size of $t$. All the implied constants are uniform in $\epsilon$ (and hence in $t$).

\begin{enumerate}
    \item {\it Small $|t|$}. Suppose $|t|\leq X^{-\delta}$, for some small $\delta$. Using $\exp (i \theta) = 1 + O(\theta)$ and the bounds for $\inprod{\sa \gamma \sigmab^{-1}}{\alpha}$ provided by Theorem \ref{bound}, we obtain
    \begin{align*}
        \e(\exp(i t A_\gamma)) &=1+O\lr{\frac{t}{ \# \T(X) \sqrt{\log X}}\sum_{\gamma \in \T(X)}|\inprod{\sa \gamma \sigmab^{-1}}{\alpha}|} \\
        &= 1+ O\lr{t \sqrt{\log X}} \ .
    \end{align*}
    Also, when $|t| \leq X^{-\delta}$, we see that
    \begin{equation*}
        e^{-t^2/2}=1-\frac{t^2}{2}+ O(t^4) = 1+ O \lr{t X^{-\delta}} \ .
    \end{equation*}
    Therefore
    \begin{align*}
        \int_{|t|\leq X^{-\delta}} \left | \frac{e^{-t^2/2}-\e(\exp(i t A_\gamma))}{t}\right | dt \ll \int_{|t|\leq X^{-\delta}} \sqrt{\log X} \  dt \ll X^{-\delta/2} \ .
    \end{align*}
    
    \item {\it Medium $|t|$}. Suppose $X^{-\delta} \leq |t| \leq (\log X)^{\delta}$, where $\delta>0$. Using that $s_0(\epsilon)=2-\pi^2 C_{\alpha} \epsilon^2 + O(\epsilon^3)$, we see that
    \begin{align*}
        \e(\exp(it A_{\gamma})) &= \frac{2 X^{2s_0(\epsilon)-4}}{s_0(\epsilon)}\lr{1+O(\epsilon)} \\
        &= \exp \lr{\log X(-2 \pi^2 C_{\alpha} \epsilon^2 + O(\epsilon^3)}(1+O(\epsilon))\\
        &= e^{-t^2/2}(1+O(\epsilon^3 \log X)+ O(\epsilon))\\
        &=e^{-t^2/2}+ O \lr{\frac{e^{-t^2/2}|t|^3}{\sqrt{\log X}} + \frac{e^{-t^2/2}}{(\log X)^{1/2-\delta}}} \ .
    \end{align*}
    Hence the contribution from such $t$ is
    \begin{align*}
         \int_{X^{-\delta}< |t| < (\log X)^\delta} \left | \frac{e^{-t^2/2}-\e(\exp(i t A_\gamma))}{t}\right | dt &\ll \int_{X^{-\delta}< |t| < (\log X)^\delta} \lr{\frac{e^{-t^2/2}t^2}{(\log X)^{1/2}} + \frac{e^{-t^2/2}}{|t|(\log X)^{1/2-\delta}}   } dt \\
         &\ll (\log X)^{-1/2 + \delta} \ .
    \end{align*}
    
    \item {\it Large $|t|$}. Suppose $(\log X)^{\delta} \leq |t| \leq  (\log X)^{1/2-\delta}$. Similarly as in the previous case, 
    \begin{align*}
        \e(\exp(it A_{\gamma})) \ll e^{-t^2/2+O(|t|^3 (\log X)^{-1/2})} \ll e^{-t^2/4} \ll e^{-(\log X)^{\delta/2}} \ .
    \end{align*}
    Therefore, the contribution from large $|t|$ is bounded by
    \begin{align*}
        \int_{(\log X)^{\delta} \leq |t| \leq  (\log X)^{1/2-\delta}} \frac{e^{-(\log X)^{\delta/2}}}{|t|} dt \ll (\log X)^{-1/2} \ .
    \end{align*}

\end{enumerate}

Putting everything together, we conclude the result in Theorem \ref{main theorem}(a). Parts (b) and (c) of Theorem \ref{main theorem}, where the results about the first and second moments are stated, follow easily from Corollary \ref{moments}. 

\section{Results for quadratic imaginary fields}

So far, we have described our results for the general case of a Kleinian group $\Gamma$. In this section we apply our results to Bianchi groups and their congruence subgroups. Let $K$ a quadratic number field with discriminant $d_K$. The arithmetic properties the groups $\mathrm{PSL}_2(\OK)$ and their congruence subgroups, as well as the geometry of the corresponding quotient spaces, are thoroughly described in \cite[Chapter 7]{egm}, while the theory of Eisenstein series for $\Gamma=\mbox{PSL}_2(\OK)$ is developed in \cite[Chapter 8]{egm}.

The ring of integers $\OK$ has the $\z$-basis consisting of $1$ and $\omega$, where
\begin{equation*}
    \omega = \frac{d_K + \sqrt{d_K}}{2} \ .
\end{equation*}
We denote by $\mathcal{P}_K$ a fundamental domain for this lattice.

The zeta function $\zeta_K(s)$ of $K$ is for $\Re(s)>1$ defined by
\begin{equation*}
    \zeta_K(s)=\sum_{\mathfrak{a}} \frac{1}{N(\a)^s} \ ,
\end{equation*}
where the sum is over the non-zero ideals of $\OK$ and the norm of $\a$ is $N(\a)=|\OK / \a|$.

As mentioned in the introduction, Cremona has several results about modular symbols associated to quadratic imaginary number fields. He uses them to compute spaces of modular forms and to establish an arithmetic correspondence between elliptic curves and cusp forms, see \cite{cremona_thesis}, \cite{cremona2}, \cite{cremona}. For consistency reasons, we will use the notation used in his work.

For technical reasons, we assume the $K$ has class number one. This is not a vital restriction, but it allows us to obtain nice arithmetic descriptions of the cusps and easier formulae relating modular symbols to $L$-functions. Let $\mathfrak{n}$ be a nonzero ideal in the ring of integers $\mathcal{O}_K$. We work with the congruence subgroup
\begin{equation*}
    \Gamma_0(\mathfrak{n}) := \left \{ \begin{pmatrix} a & b \\ c & d\end{pmatrix}\in \mbox{PSL}_2(\OK)  \ : \ c \in \mathfrak{n}\right \} \ .
\end{equation*}

A basis for the left-invariant differential 1-forms on $\h^3$ is chosen to be
\begin{equation}
    \label{beta}
    \beta= \lr{-\frac{dz}{y}, \frac{dy}{y}, \frac{d \overline{z}}{y}} \ .
\end{equation}

Let $F : \h^3 \to \c^3$ be a vector-valued function which we can write as $F=(F_0, F_1, F_2)$, then we define the differential 1-form
\begin{equation}
    F \cdot \beta := \frac{1}{y}\lr{-F_0 dz + F_1 dy + F_2 d \overline{z}} .
\end{equation}

\begin{definition}
Let $F: \h^3 \to \c^3$ be a vector-valued function and $\gamma \in \mathrm{GL}_2(\c)$. Then we define a new function $(F| \gamma): \h^3 \to \c^3$ by 
\begin{equation*}
    (F|\gamma)(P) : = F (\gamma P) j(\gamma; P) \ ,
\end{equation*}
where
\begin{equation*}
    \label{j}
   j(\gamma; P)= \frac{1}{|r|^2+|s|^2} \begin{pmatrix} 
r^2 & -2rs & s^2\\
r \overline{s} & |r|^2-|s|^2 & -\overline{r} s\\
\overline{s}^2 & 2 \overline{rs} & \overline{r}^2
   \end{pmatrix}
\end{equation*}
with $r=\overline{cz+d}$ and $s=\overline{c}y$. 
\end{definition}

This definition ensures that the differential $F \cdot \beta$ is invariant under $\gamma$ if and only if $F|\gamma=F$.

\begin{definition}

A cusp form of weight 2 for $\Gamma_0(\mathfrak{n})$ is a vector-valued function $F : \h^3 \to \c^3$ such that
\begin{enumerate}
    \item $F \cdot \beta$ is a harmonic 1-form;
    \item $F | \gamma =F$, for all $\gamma \in \Gamma_0(\mathfrak{n})$;
    \item For all $\gamma \in \mathrm{PSL}_2(\mathcal{O}_K)$ and $y \geq 0$, 
    \begin{equation*}
        \int_{\mathcal{P}_K}(F|\gamma)(z,y) dz =0 \ .
    \end{equation*}
\end{enumerate}
\end{definition}

We denote the space of cusp forms of weight 2 for $\Gn$ by $S(\mathfrak{n})$. We note that $F\in S(\mathfrak{n})$  if and only if $F \cdot \beta$ is a cuspidal 1-form for $X_0(\mathfrak{n}):= \Gamma_0(\mathfrak{n}) \backslash \h^*$, where $\h^{*}=\h^3 \cup K \cup \{ \infty \}$. In fact, the map
\begin{align*}
    S(\mathfrak{n}) & \to H^1_{\text{cusp}} (X_0(\mathfrak{n}), \c) \\ F & \mapsto F \cdot \beta
\end{align*}
is an isomorphism.

For $F \in S (\mathfrak{n})$, we have the Fourier expansion 
\begin{equation}
    \label{fourier expansion 2}
    F=(F_0,F_1,F_2)= \sum_{0 \neq \alpha \in \OK} c(\alpha) y^2 {\bf K} \lr{ \frac{4 \pi |\alpha| y}{\sqrt{|d_K|}}} \psi \lr{\frac{\alpha z}{\sqrt{d_K}}}
\end{equation}
where $\psi(z)= e(z+ \overline{z})$ and 
\begin{equation*}
    {\bf K}(y)= \lr{ -\frac{i}{2} K_1(y), K_0(y), \frac{i}{2} K_1(y)}
\end{equation*}
for $y>0$ and $K_0$, $K_1$ the $K$-Bessel functions.

The theory of cusp forms and associated $L$-functions, Hecke operators, newforms etc. is similar to the classical Atkin--Lehner theory over $\q$. We briefly recall the elements we need for our exposition.

For primes $\mathfrak{\pi}$ in $\OK$ which do not divide the level $\mathfrak{n}$, the Hecke operator $T_{\pi}$ sends the cusp form with Fourier coefficients $c(\alpha)$ to one with coefficients $c'(\alpha)$, where $c'(\alpha)=N(\pi) (\alpha \pi) + c(\alpha / \pi)$, where $c(\alpha)=0$ if $\alpha \not \in \OK$. As in the classical case, a newform in $S(\mathfrak{n})$ is an eigenform for all Hecke operators $T_{\pi}$, for $\pi$ not dividing $\mathfrak{n}$, which is not induced by a form in $S(\mathfrak{m})$, for any level $\mathfrak{m}$ properly dividing $\mathfrak{n}$. 

Secondly, let $\mathfrak{e}$ a divisor of $\mathfrak{n}$ and $e$ is a generator for $\mathfrak{e}$. Then the Atkin--Lehner operator $W_{\mathfrak{e}}$  on $S(\mathfrak{n})$ is given by the action of any matrix of the form $ \begin{pmatrix} ae & b \\ cN & de \end{pmatrix}$ which has determinant $e$. Then this operator is an involution and it commutes with the action of all Hecke operators.

Let $\epsilon$ be a unit in $\OK^*$ and $I_{\epsilon}$ denote the matrix $\begin{pmatrix} \epsilon & 0 \\ 0 & 1\end{pmatrix}$. The action of $I_{\epsilon}$ on $\h^3$ sends $(z,y)$ to $(\epsilon z, y)$ and if $F \in S(\mathfrak{n})$ has Fourier coefficients $c(\alpha)$, then $F|I_{\epsilon}$ has Fourier coefficients $c(\epsilon \alpha)$. Since $\begin{pmatrix} \epsilon^2 & 0 \\ 0 & 1\end{pmatrix}$ and $\begin{pmatrix} \epsilon & 0 \\ 0 & \epsilon^{-1} \end{pmatrix}$ give birth to the same action, but the latter belongs to $\Gn$, we must have that $c(\alpha)=c(\epsilon^2 \alpha)$, for all units $\epsilon \in \OK^*$. Hence if $\epsilon$ is a generator for the unit group $\OK^*$, then $I_{\epsilon}$ induces an involution of $S(\mathfrak{n})$ which commutes with the Hecke operators, hence we can split $S(\mathfrak{n})$ into two eigenspaces
\begin{equation*}
    S(\mathfrak{n})=S^{+}(\mathfrak{n}) \oplus S^{-}(\mathfrak{n}) \ .
\end{equation*}
Newfroms in $S^+(\mathfrak{n})$ are called plusforms, and their Fourier coefficients satisfy $c(\alpha)=c(\epsilon \alpha)$, for all $\alpha \in \OK^*$. Hence they depend only on the ideal $(\alpha)$. So if $F \in S^+ (\mathfrak{n})$, we attach to $F$ the $L$-function
\begin{equation*}
    L(F, s)= \sum_{\a} \frac{c(\a)}{N(\a)^s} \ .
\end{equation*}
Since the Fourier coefficients $c(\a)$ are multiplicative, we obtain the Euler product
\begin{equation*}
    L(F,s) = \prod_{\mathfrak{p}}(1-c(\mathfrak{p}) N(\mathfrak{p})^{-s} + \chi(\mathfrak{p})N(\mathfrak{p})^{1-2s})^{-1}, \quad \text{where } \chi(\mathfrak{p})=\begin{cases} 0 & \text{if } \mathfrak{p} \mid \mathfrak{n} \ , \\ 1 & \text{if } \mathfrak{p} \nmid \mathfrak{n} \ . \end{cases}
\end{equation*}

Similar to classical case, one can deduce the Ramanujan bound $|c(\mathfrak{p})| \leq 2 N(\mathfrak{p})^{1/2}$, from which it follows that $L(F,s)$ converges for $\Re(s) > 3/2$.

We now consider additive twists of this $L$-function. Fix $r=a/c \in K$. If $F$ is a plusform, then we define $L(F,s,r)$ as
$$L(F,s,r):=\sum_{0 \neq \alpha \in \OK} \frac{c(\alpha)}{N (\alpha) ^ s} \psi \lr{\frac{\alpha r}{\sqrt{d_K}}} = \sum_{(\alpha)} \frac{c((\alpha))}{N ((\alpha)) ^ s} \widetilde{\psi}  \lr{\frac{\alpha r}{\sqrt{d_K}}} $$
where the second sum is over all ideals $(\alpha)$ and $$\widetilde{\psi}(z):=\frac{1}{|\OK^*|}\sum_{\epsilon \in \OK^*}\psi(\epsilon z)$$ is invariant over generators of an ideal.

We form the Mellin transform of $F$ by multiplying by $y^{2s-2}$ and integrating along a vertical imaginary axis. For $s \in \c$ and $r \in K$, we define
\begin{equation*}
    \Lambda(F, s, r): =\int_{r}^{j \infty} y^{2s-2} F\cdot\beta = \int_{0}^{\infty} y^{2s-2} F_1 (r,y) \frac{dy}{y} \ . 
\end{equation*}
The rapid decay of $F(z,y)$ in the cusps ensures that $\Lambda(F,s,r)$ is an entire function of $s \in \c$. 
 
 We note that if $F$ is a plusform, then we can write modular symbols as central values of twisted $L$-function:
 \begin{equation}
     \langle r \rangle = \Lambda(F, 1, r) = \int_{r}^{\infty} F \cdot\beta  \ .
 \end{equation}

We obtain analytic continuation and functional equation for $L(F,s,r)$. 
\begin{lemma}
   Let $F$ be a plusform in $S(\mathfrak{n})$. Then
   \begin{itemize}
       \item[(a)] For $\Re(s)>3/2$, we have
       \begin{equation*}
           \Lambda(F,s,r)= \frac{1}{4} \lr{\frac{|c| \sqrt{|d_K|}}{2 \pi}}^{2s} \Gamma(s)^2 \ L(F, s, r ) \ .
       \end{equation*}
       \item[(b)] Write $\mathfrak{n}=\mathfrak{e} \mathfrak{f}$, where $\mathfrak{f}=\mathfrak{n}+ (c)$. Let $\mathfrak{e}=(e)$. Denote by $w_e$ the eigenvalue of the Fricke involution $W_e$ acting on $F$. Then we have the following functional equation:
       \begin{equation*}
           \Lambda(F,s,a/c)=- w_e N(\mathfrak{e})^{1-s} \Lambda \left (F, 2-s, - \frac{\overline{ea}}{c} \right ) \ ,
       \end{equation*}
       where $\overline{ea}$ is the inverse of $ea$ in $(\OK / (c))^*$.
       \item[(c)] With the same notation, we have $\langle a/c \rangle = - w_{\epsilon} \langle -\overline{ea}/c \rangle$.
   \end{itemize}
\end{lemma}

We now quote \cite[p. 415]{cremona} and note that if $F$ is a plusform in $S_2(\mathfrak{n})$, then the image of the map
\begin{align*}
    I_F:\Gamma_0(\mathfrak{n})  \to \c \ , \quad
    I_F(\gamma) =  \int_{A}^{\gamma A} F \cdot \beta
\end{align*}
is a discrete, nontrivial subgroup of $\r$, hence of the form $\Omega(F) \z$, for some real $\Omega(F)$. In \cite{cremona}, Cremona provides an algorithm for computing $\Omega(F)$. We show that for a fixed newform $F$, the values in the image of the map $I_F$ are normally distributed with the required normalisation and ordering.

We have the following description of equivalent $\Gamma_0(\mathfrak{n})$-equivalent points in $K$, as in \cite[Proposition 4.2.2]{cremona_thesis} or \cite[Lemma 2.2.7]{cremona2}:

\begin{proposition}
 Let $\frac{p_1}{q_1}, \frac{p_2}{q_2} \in K$ be written in their lowest terms. The following are equivalent:
 \begin{enumerate}
     \item There exists $\gamma \in \Gamma_0(\mathfrak{n})$ such that $\gamma \lr{ \frac{p_1}{q_1}}= \frac{p_2}{q_2}$;
     \item There exists $u \in \OK^*$ such that $s_1 q_2 \equiv u^2 s_2 q_1 (\mbox{mod } (q_1q_2) + \mathfrak{n})$, where $p_k s_k \equiv 1 (\mbox{mod } (q_k))$, for $k=1,2$.
 \end{enumerate}
\end{proposition}

Hence we can provide the following description for the inequivalent cusps for $\Gamma_0(\mathfrak{n})$, where $\mathfrak{n}$ is square-free. For each ideal $\mathfrak{d} | \mathfrak{n}$, we fix some $d \in \OK$ such that $(d)= \mathfrak{d}$. Then a complete set of inequivalent cusps are given by $a_{\mathfrak{d}}=1/d$ with $\mathfrak{d} | \mathfrak{n}$. If $\mathfrak{d}=\mathfrak{n}$, then $1/d$ is equivalent to the cusp at infinity. Moreover,
\begin{equation*}
    R_{\infty \mathfrak{d}}= \left \{ \frac{a}{c} \textrm{ mod } \mathcal{P}_K \ : \ a \in (\OK / (c))^{*}, (c) + \mathfrak{n} = \mathfrak{d} \right \} \ .
\end{equation*}
and 
\begin{equation*}
    \rr_{\infty \mathfrak{d}} = \int_{1/d}^r F \cdot \beta = \int_{1/d}^{j \infty} F \cdot \beta + \rr .
\end{equation*}

Also, for all cusps $\mathfrak{d}$, we have that $[\Gamma_{\mathfrak{d}}:\Gamma_{\mathfrak{d}}']=|\OK^*|/2$. In particular, $|\mathcal{O}_{\q(i)}^*|=4$, $|\mathcal{O}_{\q(\sqrt{-3})}^*|=6$ and $|\mathcal{O}_K^*|=2$ for all other quadratic imaginary number fields. 

We note that we now have all the ingredients to derive Corollary \ref{quadratic} from Theorem \ref{main theorem}. Indeed, from \cite[Theorem 6.1.1]{egm} we see that the covolume of $\textrm{PSL}_2(\OK)$ is
\begin{equation*}
    \textrm{vol}(\textrm{PSL}_2(\OK)) = \frac{|d_K|^2}{4 \pi^2} \zeta_K(2) 
\end{equation*}
and similarly as in the 2-dimensional case, we can deduce
\begin{equation*}
    [\textrm{PSL}_2(\OK): \Gn]= \prod_{\mathfrak{p}|\mathfrak{n}}(1+|\mathfrak{p}|) \ .
\end{equation*}
Finally, the Petersson norm of $F$ is given by
\begin{equation*}
    \| F \|^2 = \inprod{F  \cdot \beta}{ F \cdot \beta} = \int_{\GH}(2 |F_1|^2 + |F_2|^2 + 2|F_3|^2) d v \ .
\end{equation*}
Putting all together, we deduce that the constant $C_F$ in Corollary \ref{quadratic} is given by
\begin{equation}
    \label{C_F}
    C_F = \frac{4 \pi^2 \| F \|^2}{|d_K|^2 \zeta_K(2) \prod_{\mathfrak{p}|\mathfrak{n}}(1+|\mathfrak{p}|)  } \ . \\
\end{equation}

{\bf Acknowledgements} We would like to thank Yiannis Petridis for suggesting the problem to us and for his advice, patience and many useful discussions. This work was supported by the Engineering and Physical Sciences Research Council [EP/L015234/1]. The EPSRC Centre for Doctoral Training in Geometry and Number Theory (The London School of Geometry and Number Theory), University College London.

\end{document}